\documentclass[reqno,11pt,dvips]{article}
\usepackage{amsmath,amsfonts,amssymb,amsthm,enumerate,graphicx,ifthen,bm,needspace,fancyhdr,multirow,array}
\usepackage[small,bf,hang]{caption}
\usepackage{geometry}
\usepackage{fullpage}
\usepackage[bookmarks=true,implicit=true,pagebackref=false]{hyperref}
\usepackage[all]{hypcap}

\newcommand{\reaction}[1]{\overset{#1}{\rightarrow}}
\newcommand{\revreaction}[2]{\overset{#1}{\underset{#2}{\rightleftharpoons}}}

\newcommand{\iterateD}[3]{\vecchi^{\rb{#1}}\!\rb{#2,#3}}
\newcommand{\iterateDi}[4]{\chi_{#1}^{\rb{#2}}\!\rb{#3,#4}}
\newcommand{\iterateDj}[4]{\vecchi_{#1}^{\rb{#2}}\!\rb{#3,#4}}

\newcommand{\psiif}[2]{\psi_{#1}\!\rb{#2}}

\newcommand{\vecpsif}[1]{\vecpsi\!\rb{#1}}

\newcommand{\vecpsiif}[2]{\vecpsi_{#1}\!\rb{#2}}
\newcommand{\vecpsimf}[2]{\vecpsi^{\rb{#1}}\!\rb{#2}}
\newcommand{\vecpsiimf}[3]{\vecpsi_{#1}^{\rb{#2}}\!\rb{#3}}

\newcommand{\flowB}[2]{\vecphi_{#1}\!\rb{#2}}
\newcommand{\flowC}[2]{\vecphi\!\rb{#1,#2}}
\newcommand{\flowCi}[3]{\phi_{#1}\!\rb{#2,#3}}

\newcommand{\flowE}[2]{\vecvarphi\!\rb{#1,#2}}
\newcommand{\flowEj}[3]{\vecvarphi_{#1}\!\rb{#2,#3}}

\newcommand{\scf}[2]{#1\!\rb{#2}}
\newcommand{\scif}[3]{#1_{#2}\!\rb{#3}}
\newcommand{\vecf}[2]{\vec{#1}\!\rb{#2}}
\newcommand{\vecif}[3]{\vec{#1}_{#2}\!\rb{#3}}
\newcommand{\vecmf}[3]{\vec{#1}^{\rb{#2}}\!\rb{#3}}
\newcommand{\vecinvf}[2]{\vec{#1}^{\,-1}\!\rb{#2}}

\newcommand{\lambdam}{\lambda_{_-}}
\newcommand{\lambdap}{\lambda_{_+}}
\newcommand{\lambdapm}{\lambda_{_\pm}}
\newcommand{\sigmam}{\sigma_{_-}}
\newcommand{\sigmap}{\sigma_{_+}}
\newcommand{\sigmapm}{\sigma_{_\pm}}

\newcommand{\mymbox}[1]{\quad\mbox{#1}\quad}
\newcommand{\fracslash}[2]{{#1}/{#2}}
\newcommand{\sfrac}[2]{\textstyle\frac{#1}{#2}\displaystyle}
\newcommand{\ds}{\displaystyle{}}
\newcommand{\dsm}[1]{$\ds{#1}$}
\newcommand{\vep}{\varepsilon}
\newcommand{\st}{\,:\,}
\newcommand{\shalf}{\sfrac{1}{2}}

\newcommand{\rb}[1]{\left(#1\right)}
\newcommand{\sqb}[1]{\left[#1\right]}
\newcommand{\setb}[1]{\left\{#1\right\}}
\newcommand{\Cr}[2]{\ifthenelse{\equal{#2}{}}{C^{#1}}{C^{#1}\!\rb{#2}}}
\newcommand{\Nd}[1]{\mathbb{N}^{#1}}
\newcommand{\Rn}{\mathbb{R}^n}
\newcommand{\Rd}[1]{\mathbb{R}^{#1}}

\newcommand{\ballO}[1]{\cB_{#1}}
\newcommand{\abs}[1]{\left|#1\right|}
\newcommand{\norm}[1]{\left\|#1\right\|}
\renewcommand{\vec}[1]{\textbf{#1}}
\newcommand{\mat}[1]{\textbf{#1}}
\newcommand{\diag}[1]{\textsf{diag}\rb{#1}}
\renewcommand{\det}[1]{\abs{#1}}
\newcommand{\integral}[4]{\int_{#1}^{#2} \, #3 \, d#4}
\newcommand{\deriv}[2]{\frac{d#1}{d#2}}
\newcommand{\derivslash}[2]{\fracslash{d#1}{d#2}}

\newcommand{\pderiv}[2]{\frac{\partial#1}{\partial#2}}

\renewcommand{\max}[2]{\underset{#1}{\textsf{max}}\setb{#2}}
\renewcommand{\min}[2]{\underset{#1}{\textsf{min}}\setb{#2}}

\newcommand{\seq}[4]{\setb{{#1}_{#2}}_{\ifthenelse{\equal{#3}{}}{}{#2=#3}}^{#4}}

\newcommand{\bigO}[1]{\cO\!\rb{#1}}
\newcommand{\litO}[1]{\textsf{o}\!\rb{#1}}

\newcommand{\lnb}[1]{\textsf{ln}\!\rb{#1}}
\renewcommand{\exp}[1]{\textsf{e}^{#1}}

\newcommand{\floor}[1]{\left\lfloor#1\right\rfloor}

\newcommand{\matLambda}{\bm{\Lambda}}

\newcommand{\vecphi}{\bm{\phi}}
\newcommand{\vecvarphi}{\bm{\varphi}}

\newcommand{\vecpsi}{\bm{\psi}}

\newcommand{\vecchi}{\bm{\chi}}

\newcommand{\vecxi}{\bm{\xi}}

\newcommand{\cB}{\mathcal{B}}
\newcommand{\cM}{\mathcal{M}}
\newcommand{\cO}{\mathcal{O}}

\newcommand{\hmu}{\widehat{\mu}}

\newcommand{\twobyonematrix}[2]{\begin{pmatrix}#1\\#2\end{pmatrix}}
\newcommand{\twobytwomatrix}[4]{\begin{pmatrix}#1&#2\\#3&#4\end{pmatrix}}

\theoremstyle{plain}
\newtheorem{thm}{\bf{Theorem}}
\newtheorem{cor}[thm]{\bf{Corollary}}
\newtheorem{lem}[thm]{\bf{Lemma}}
\newtheorem{prop}[thm]{\bf{Proposition}}
\newtheorem{claim}[thm]{\bf{Claim}}

\theoremstyle{definition}
\newtheorem{defn}[thm]{\bf{Definition}}

\theoremstyle{remark}
\newtheorem{rem}[thm]{\bf{Remark}}

\renewenvironment{proof}{\medskip\noindent\em Proof:
\rm}{\hspace*{\fill}$\square$\medskip}

\begin{document}

\pagestyle{fancy}

\renewcommand{\thefootnote}{\fnsymbol{footnote}}

\vspace{-1cm}

\title{Asymptotic Behaviour Near a Nonlinear Sink}

\author{Matt S. Calder\footnotemark[1]\\Department of Applied Mathematics,
University of Western Ontario\\1151 Richmond Street, London, Ontario, Canada, N6A 5B7 \and David
Siegel\footnotemark[2]\\Department of Applied Mathematics,
University of Waterloo\\200 University Avenue West, Waterloo,
Ontario, Canada, N2L 3G1}

\date{}

\footnotetext[1]{Research partially supported by an Ontario Graduate
Scholarship.}

\footnotetext[2]{Research supported by a Natural Sciences and
Engineering Research Council of Canada Discovery Grant.}

\renewcommand{\thefootnote}{\arabic{footnote}}

\maketitle

\begin{abstract} \addcontentsline{toc}{section}{Abstract}
In this paper, we will develop an iterative procedure to determine
the detailed asymptotic behaviour of solutions of a certain class of
nonlinear vector differential equations which approach a nonlinear
sink as time tends to infinity. This procedure is indifferent to
resonance in the eigenvalues. Moreover, we will address the writing
of one component of a solution in terms of the other in the case of a planar
system. Examples will be given, notably the Michaelis-Menten
mechanism of enzyme kinetics.
\end{abstract}



\section{Introduction} \label{sec001}

\subsection{The Problem}

Suppose we are presented with a vector ordinary differential equation with time as the independent variable. Suppose also that the origin is an asymptotically stable equilibrium point, that is, all solutions starting sufficiently close to the origin will approach the origin as time tends to infinity. It is very desirable
to obtain the asymptotic behaviour of solutions as time tends to infinity with sufficient detail when starting close to the origin. In this paper, we will develop an iterative procedure, which is indifferent to resonance in the eigenvalues, to construct asymptotic expansions of solutions with any desired accuracy. This detailed asymptotic behaviour, for example, will enable us to construct phase portraits with more detail than the linearization since the iterative procedure yields higher-order (nonlinear) approximations. Moreover, the detailed behaviour will have implications for concavity results, parameter estimation and reduction methods in chemical kinetics, and it will show the limitations of using traditional power series.

Let ${n\in\Nd{}}$ be the dimension and consider the constant matrix
${\mat{A}\in\Rd{n \times n}}$ and the vector field
${\vec{b}:\Rn\to\Rn}$. In this paper, we will be considering the
general initial value problem
\begin{equation} \label{eq001}
    \dot{\vec{x}} = \mat{A} \vec{x} + \vecf{b}{\vec{x}},
    \quad
    \vec{x}(0) = \vec{x}_0,
\end{equation}
where $t$ is time and ${\dot{}=\derivslash{}{t}}$, when $\mat{A}$
and $\vec{b}$ satisfy certain properties. Note that \eqref{eq001} is equivalent to the integral
equation
\begin{equation} \label{eq002}
    \vec{x}(t) = \exp{t\mat{A}} \vec{x}_0 + \integral{0}{t}{\exp{\rb{t-s}\mat{A}}\vecf{b}{\vec{x}(s)}}{s}.
\end{equation}
The matrix $\mat{A}$ is assumed to be Hurwitz with eigenvalues
$\seq{\lambda}{i}{1}{n}$ having respective real parts
$\seq{\mu}{i}{1}{n}$ satisfying the ordering
\begin{equation} \label{eq003}
    \mu_n \leq \mu_{n-1} \leq \cdots \leq \mu_1 < 0.
\end{equation}
For the vector field $\vec{b}$, we assume there are constants ${\delta,k_1,k_2>0}$ such that ${\vec{b}\in\Cr{1}{\ballO{\delta},\Rn}}$ with
\begin{equation} \label{eq004}
    \norm{\vecf{b}{\vec{x}}} \leq k_1 \norm{\vec{x}}^\alpha
    \mymbox{and}
    \norm{\mat{D}\vecf{b}{\vec{x}}} \leq k_2 \norm{\vec{x}}^\beta
    \mymbox{for all}
    \vec{x} \in \ballO{\delta}
\end{equation}
for some ${\alpha>1}$ and ${\beta>0}$, where
\[
    \norm{\vec{x}} := \sqrt{\sum_{i=1}^nx_i^2} \mymbox{for} \vec{x} \in \Rn,
    \quad
    \ballO{\delta} := \setb{\,\vec{x}\in\Rn\st\norm{\vec{x}}<\delta\,},
    \mymbox{and}
    \mat{D}\vecf{b}{\vec{x}} := \rb{ \pderiv{\scf{b_i}{\vec{x}}}{x_j}}_{i,j=1}^n
\]
are, respectively, the Euclidean norm (chosen arbitrarily), the open
ball of radius $\delta$ centred at the origin, and the Jacobian
matrix.

\begin{rem}
The matrix $\mat{A}$, the vector field $\vec{b}$, the constants
$\seq{\lambda}{i}{1}{n}$ and $\seq{\mu}{i}{1}{n}$, and the constants
$\alpha$ and $\beta$ retain their meaning throughout the paper.
However, the constants $\delta$, $k_1$, and $k_2$ do not. We will be
dealing with many combinations of estimates which are introduced by
saying something like ``there are \emph{some} constants
${\delta,k>0}$ such that...'' and, typically, we take ${\delta>0}$
small enough so that multiple estimates apply.
\end{rem}

We will denote by $\flowB{t}{\vec{x}_0}$ the flow of \eqref{eq001}.
Where necessary, we will write $\flowC{t}{\vec{x}_0}$ with
components $\setb{\flowCi{i}{t}{\vec{x}_0}}_{i=1}^n$. Furthermore,
define the important ratio
\[
    \kappa := \frac{\mu_n}{\mu_1}.
\]
Observe that, by virtue of \eqref{eq003}, ${\kappa \geq 1}$.  We
will need to consider two distinct cases.

\begin{defn}
If ${1\leq\kappa<\alpha}$, the eigenvalues of the matrix $\mat{A}$ are \emph{closely-spaced} relative to the vector field $\vec{b}$. If ${\kappa\geq\alpha}$, the eigenvalues of $\mat{A}$ are \emph{widely-spaced} relative to $\vec{b}$.
\end{defn}

The goal of this paper is to find the detailed asymptotic behaviour
of the flow $\flowB{t}{\vec{x}_0}$ as ${t\to\infty}$. The techniques
we will develop are indifferent to resonance in the eigenvalues of
$\mat{A}$. However, it is important to be clear what resonance is
and when it occurs. There is \emph{resonance} in the eigenvalues if
there exist ${\seq{m}{i}{1}{n}\subset\Nd{}_0}$ and
${j\in\setb{1,\ldots,n}}$ with ${\sum_{i=1}^n m_i \geq 2}$ and
${\lambda_j = \sum_{i=1}^n m_i \lambda_i}$. The \emph{order} of the
resonance is ${\sum_{i=1}^n m_i}$. It is easy to verify that
resonance can only occur if ${\kappa \geq 2}$.

In the remainder of this section, we provide some background
information and establish important estimates. In \S\ref{sec002}, we
focus on the case where the eigenvalues are closely-spaced relative
to the nonlinear part. To approximate the solution
$\flowB{t}{\vec{x}_0}$, we will construct iterates
$\setb{\iterateD{m}{t}{\vec{y}_0}}_{m=1}^\infty$. The difference (in
norm) of the first iterate $\iterateD{1}{t}{\vec{y}_0}$ and the flow
$\flowB{t}{\vec{x}_0}$ is, for any ${\sigma>0}$,
$\bigO{\exp{\alpha\rb{\mu_1+\sigma}t}\norm{\vec{x}_0}^\alpha}$ as
${t\to\infty}$ and ${\norm{\vec{x}_0} \to 0}$. Moreover, each
iteration increases the closeness by a factor which is
$\bigO{\exp{\beta\rb{\mu_1+\sigma}t}\norm{\vec{x}_0}^\beta}$. In
\S\ref{sec003}, we focus on the special case where ${n=2}$,
$\mat{A}$ is diagonalizable, and the eigenvalues are widely-spaced
relative to the nonlinear part. We construct iterates, similar to
the closely-spaced case, to approximate the flow. We also develop a
result on the expression of the second component
$\flowCi{2}{t}{\vec{x}_0}$ of the flow in terms of the first
component $\flowCi{1}{t}{\vec{x}_0}$ when there is resonance and a
quadratic nonlinearity. The general $n$-dimensional case is
presented as well. In \S\ref{sec004}, we apply our results to the
Michaelis-Menten mechanism of an enzyme-substrate reaction to
generalize a result in the authors' \cite{CalderSiegel}. Finally, in
\S\ref{sec005}, we state some open questions.

\subsection{Background}

The two main contributions of Poincar\'{e} relevant to this paper
are normal forms, which originated in his Ph.D. thesis, and a result
often referred to simply as Poincar\'{e}'s Theorem, which he proved
in 1879. These are contained in the first volume (of eleven) of the
\emph{Oeuvres} of Poincar\'{e} \cite{Poincare}.

Normal-form theory involves the substitution of an
analytic, near-identity transformation to simplify the nonlinear
part of an ordinary differential equation. Resonance in the eigenvalues of $\mat{A}$ limits how simple the nonlinear part can be made. See, for example,
\cite{BronsteinKopanskii,ChowLiWang,Perko,Wiggins}.

Consider the two systems
${\dot{\vec{x}}=\mat{A}\vec{x}+\vecf{b}{\vec{x}}}$ and
${\dot{\vec{y}}=\mat{A}\vec{y}}$, where $\vec{b}$ is analytic in a
neighbourhood of the origin and satisfies
${\norm{\vecf{b}{\vec{x}}}=\bigO{\norm{\vec{x}}^2}}$ as
${\norm{\vec{x}} \to 0}$. If $\mat{A}$ has non-resonant eigenvalues
and the convex hull of the eigenvalues does not contain zero, then
Poincar\'{e}'s Theorem says that there is a near-identity,
invertible transformation $\vec{h}$ with both $\vec{h}$ and
$\vec{h}^{-1}$ analytic at the origin such that
${\vecf{h}{\flowB{t}{\vec{x}_0}}=\exp{t\mat{A}}\vecf{h}{\vec{x}_0}}$.
That is, the flows of the two systems are analytically conjugate in
a neighbourhood of the origin. This helps greatly in determining
qualitative properties of solutions of more complicated systems.
Furthermore, if we write
${\vecinvf{h}{\vec{y}}=\vec{y}+\vecf{r}{\vec{y}}}$, then the
relationship
${\flowB{t}{\vec{x}_0}=\exp{t\mat{A}}\vecf{h}{\vec{x}_0}+\vecf{r}{\exp{t\mat{A}}\vecf{h}{\vec{x}_0}}}$
can be used to extract asymptotic expansions.

The Hartman-Grobman Theorem is a conjugacy result similar to
Poincar\'{e}'s Theorem. See, for example,
\cite{BronsteinKopanskii,ChiconeSwanson,Grobman,Hartman1,Hartman2,Hartman3}.
See also \cite{Bonckaert} which addresses the issue of just how
small the neighbourhoods of the origin must be. Conjugacy results
such as the Hartman-Grobman Theorem (and its variations) and
Poincar\'{e}'s Theorem are distinguished by the specifics of the
system: The smoothness of $\vec{b}$, the smoothness of $\vec{h}$,
whether or not the eigenvalues of $\mat{A}$ are resonant, and
whether or not the origin is a hyperbolic fixed point.

The standard methods have well-known limitations.  Many results have
to exclude resonance and those that do not are often limited in
their practicality. Leading-order behaviour of solutions can be
difficult to obtain and more detailed behaviour can be more
difficult still. Finally, the standard methods often give little or
no detail on the relationship between components of solutions.

In this paper, we will frequently encounter a certain class of
functions defined by improper integrals.  Let ${\delta>0}$ be a
constant and define the set
\[
    \Omega_\delta := \setb{ \, \rb{t,\vec{x}} \st t \geq 0, \, \vec{x} \in \Rn, \, \norm{\vec{x}} < \delta \, }.
\]
Let ${\vec{f}\in\Cr{}{\Omega_\delta,\Rn}}$ and suppose there are
${k,\rho>0}$ such that ${\norm{\vecf{f}{t,\vec{x}}} \leq k \,
\exp{-\rho t}}$ for every ${\rb{t,\vec{x}}\in\Omega_\delta}$. Then,
${\vecf{g}{t,\vec{x}}:=\integral{t}{\infty}{\vecf{f}{s,\vec{x}}}{s}}$,
defined for ${\rb{t,\vec{x}}\in\Omega_\delta}$, satisfies
${\norm{\vecf{g}{t,\vec{x}}}\leq\fracslash{k}{\rho}}$ for all
${\rb{t,\vec{x}}\in\Omega_\delta}$. It follows from the Weierstrass
$M$-Test that the integral converges uniformly on $\Omega_\delta$
and ${\vec{g}\in\Cr{}{\Omega_\delta,\Rn}}$. See, for example,
Theorems~14-19 and 14-22 of \cite{Apostol} and \S6.5 of
\cite{LevinsonRedheffer}.

\subsection{Important Exponential Estimates}

\begin{claim} \label{claim001}
For ${\matLambda:=\diag{\lambda_1,\ldots,\lambda_n}}$, ${\norm{\exp{t\matLambda}}=\exp{\mu_1t}}$ and ${\norm{\exp{-t\matLambda}}=\exp{-\mu_nt}=\exp{-\kappa\mu_1t} }$ for every ${t \geq 0}$.
\end{claim}

\begin{proof}
The proof is straight-forward and omitted.
\end{proof}

\begin{lem}
Consider the matrix $\mat{A}$. For every ${\sigma>0}$, there exist ${k_1,k_2>0}$ such that
\begin{equation} \label{eq005}
    \norm{ \exp{t\mat{A}} } \leq k_1 \, \exp{ \rb{ \mu_1 + \sigma } t }
    \mymbox{and}
    \norm{  \exp{-t\mat{A}} } \leq k_2 \, \exp{\rb{-\mu_n+\sigma}t} = k_2 \, \exp{\rb{-\kappa\mu_1+\sigma}t} \mymbox{for all} t \geq 0.
\end{equation}
Moreover, if $\mat{A}$ is diagonalizable, then we may take ${\sigma=0}$ in \eqref{eq005}.
\end{lem}

\begin{proof}
The proofs of the estimates when ${\sigma>0}$ are standard. Assume $\mat{A}$ is diagonalizable. That is, there is an invertible matrix $\mat{P}$ such that
${\mat{A}=\mat{P}\matLambda\mat{P}^{-1}}$, where ${\matLambda:=\diag{\lambda_1,\ldots,\lambda_n}}$. Thus, ${\exp{t\mat{A}}=\mat{P}\exp{t\matLambda}\mat{P}^{-1}}$. The
estimates when ${\sigma=0}$ then follow from Claim~\ref{claim001}.
\end{proof}

\begin{thm}
Consider the flow $\flowB{t}{\vec{x}_0}$ for \eqref{eq001}. For any ${\sigma\in\rb{0,-\mu_1}}$, there exist ${\delta,k>0}$
such that
\begin{equation} \label{eq006}
    \norm{ \flowB{t}{\vec{x}_0} } \leq k \, \exp{\rb{\mu_1+\sigma}t} \norm{\vec{x}_0}
    \mymbox{for all}
    \rb{t,\vec{x}_0} \in \Omega_\delta.
\end{equation}
\end{thm}

\begin{proof}
We will assume that ${\norm{\vec{x}_0}>0}$ for if this were not the
case the result would be trivial. This proof is more or less
standard but we are careful to include the initial condition in the
estimate which is not standard. See, for example, pages~314 and 315
of \cite{CoddingtonLevinson}. In \cite{CoddingtonLevinson}, they
used Gronwall's Inequality at a key step (they had a slight
difference in assumptions) but we will instead use a modification.
Such modifications are examples of Bihari's Inequality.  See, for
example, \cite{Bihari}.

It follows from the integral equation \eqref{eq002} and the estimates \eqref{eq004} and \eqref{eq005} that
\begin{equation} \label{eq007}
    \norm{ \flowB{t}{\vec{x}_0} }
    \leq k_1 \exp{ \rb{\mu_1+\sigma} t } \norm{ \vec{x}_0 } + k_2 \integral{0}{t}{ \exp{ \rb{\mu_1+\sigma} \rb{t-s} } \norm{ \flowB{s}{\vec{x}_0} }^{\alpha} }{s}
\end{equation}
for some ${\delta_1,k_1,k_2>0}$, provided ${\norm{\flowB{t}{\vec{x}_0}}<\delta_1}$ for all ${t \geq 0}$. Later, we will show this condition is satisfied for
$\norm{\vec{x}_0}$ sufficiently small. Manipulate \eqref{eq007} to obtain ${u(t) \leq v(t)}$, where
\[
    u(t) := \exp{ -\rb{\mu_1+\sigma} t } \norm{ \flowB{t}{\vec{x}_0} } > 0
    \mymbox{and}
    v(t) := k_1 \norm{ \vec{x}_0 } + k_2 \integral{0}{t}{ \exp{ \rb{\alpha-1} \rb{\mu_1+\sigma} s } u(s)^{\alpha} }{s} > 0.
\]
Now, ${\dot{v}(t)=k_2\,\exp{\rb{\alpha-1}\rb{\mu_1+\sigma}t}u(t)^{\alpha}>0}$. Since ${u(t) \leq v(t)}$ and ${v(t)>0}$, we have
\begin{equation} \label{eq008}
    \frac{ \dot{v}(s) }{ v(s)^{\alpha} } \leq k_2 \, \exp{ \rb{\alpha-1} \rb{\mu_1+\sigma} s },
    \quad
    v(0) = k_1\norm{\vec{x}_0}
    \mymbox{for}
    s \geq 0.
\end{equation}
If we integrate \eqref{eq008} with respect to $s$ from $0$ to $t$ and recall that
${\alpha>1}$ and ${\mu_1+\sigma<0}$, we obtain
\[
    \frac{1}{1-\alpha} \sqb{ \frac{1}{v(t)^{\alpha-1}} - \frac{1}{ k_1^{\alpha-1} \norm{ \vec{x}_0 }^{\alpha-1} } }
    \leq \sqb{ \frac{ k_2 }{ \rb{\alpha-1} \rb{\mu_1+\sigma} } } \sqb{ \exp{ \rb{\alpha-1} \rb{\mu_1+\sigma} t } - 1 }
    \leq -\frac{ k_2 }{ \rb{\alpha-1} \rb{\mu_1+\sigma} }.
\]
Rearranging,
\begin{equation} \label{eq009}
    \frac{1}{v(t)^{\alpha-1}} \geq \frac{1}{\norm{\vec{x}_0}^{\alpha-1}} \rb{ \frac{1}{ k_1^{\alpha-1} } + \frac{k_2 \norm{\vec{x}_0}^{\alpha-1}}{\mu_1+\sigma} }.
\end{equation}
Now, we want the expression in the brackets to be strictly positive,
which yields the condition
\[
    \norm{\vec{x}_0} < \frac{1}{k_1} \rb{ -\frac{ \mu_1 + \sigma }{ k_2 } }^{\frac{1}{\alpha-1}}.
\]
Take
\[
    \delta := \frac{\rho}{k_1} \rb{ -\frac{ \mu_1 + \sigma }{ k_2 } }^{\frac{1}{\alpha-1}}
    \mymbox{and}
    k
    := \rb{ \frac{1}{ k_1^{\alpha-1} } + \frac{k_2 \delta^{\alpha-1}}{\mu_1+\sigma} }^{\frac{1}{1-\alpha}}
    = k_1 \rb{ 1 - \rho^{\alpha-1} }^{\frac{1}{1-\alpha}}
    > 0,
\]
where ${\rho\in\rb{0,1}}$ is arbitrary, and assume
${\norm{\vec{x}_0}<\delta}$. Note $k$ is independent of $\norm{\vec{x}_0}$. From \eqref{eq009},
\[
    \frac{1}{v(t)^{\alpha-1}} \geq \frac{1}{ \norm{\vec{x}_0}^{\alpha-1} k^{\alpha-1} }.
\]
Rearranging and recalling that ${u(t) \leq v(t)}$, we have ${u(t)
\leq k \norm{\vec{x}_0}}$. By definition of $u(t)$, we have
${\norm{\flowB{t}{\vec{x}_0} } \leq
k\,\exp{\rb{\mu_1+\sigma}t}\norm{\vec{x}_0}}$ for all ${t \geq 0}$.

All that remains is to show that the condition ${\norm{\flowB{t}{\vec{x}_0}}<\delta_1}$ for all ${t \geq 0}$ is satisfied for a
particular choice of $\rho$. Using our expressions for $k$ and $\delta$, we see that we can choose ${\rho\in(0,1)}$
sufficiently close to zero to ensure ${k\delta<\delta_1}$. Hence, if ${\rb{t,\vec{x}_0}\in\Omega_\delta}$ then ${\norm{\flowB{t}{\vec{x}_0}} \leq k\,\exp{\rb{\mu_1+\sigma}t}\norm{\vec{x}_0}<k\delta<\delta_1}$.
\end{proof}

\begin{rem}
A glance at the proof of the previous theorem shows that $\sigma$ is
only necessary in the basic decay rate of $\flowB{t}{\vec{x}_0}$ if
$\sigma$ is necessary in the decay rate of $\exp{t\mat{A}}$. Indeed,
if $\mat{A}$ is diagonalizable then we can take ${\sigma=0}$ in
\eqref{eq006}.
\end{rem}

\section{Closely-Spaced Eigenvalues} \label{sec002}

Consider once again the system \eqref{eq001}.  In this section, we
will assume that the eigenvalues $\seq{\lambda}{i}{1}{n}$ of
$\mat{A}$ are closely-spaced relative to the nonlinear part
$\vecf{b}{\vec{x}}$. That is, in addition to \eqref{eq003} and
\eqref{eq004} holding, we are assuming
\begin{equation} \label{eq010}
    \kappa < \alpha.
\end{equation}

\begin{rem}
Throughout this section, if $\mat{A}$ is diagonalizable we can take
${\sigma=0}$.
\end{rem}

\begin{rem}
Consider the scalar case ${n=1}$ with ${\dot{x}=ax+\scf{b}{x}}$,
where ${a<0}$, so that ${\mat{A}=\rb{a}}$ and ${\kappa=1<\alpha}$.
Trivially, the results of this section apply and we can take
${\sigma=0}$ for all estimates.
\end{rem}

\subsection{The Transformation}

\begin{claim}
For any ${\sigma>0}$ there are ${\delta,k>0}$ such that
\begin{equation} \label{eq011}
    \norm{ \exp{-t\mat{A}} \vecf{b}{\flowB{t}{\vec{x}_0}} } \leq k \, \exp{ \sqb{ \rb{ \alpha - \kappa } \mu_1 + \sigma } t } \norm{ \vec{x}_0 }^{\alpha}
    \mymbox{for all}
    \rb{t,\vec{x}_0} \in \Omega_\delta.
\end{equation}
\end{claim}

\begin{proof}
It follows from \eqref{eq004}, \eqref{eq005}, and \eqref{eq006}.
\end{proof}

Recall the integral equation \eqref{eq002}. When $\norm{\vec{x}_0}$
is sufficiently small, \eqref{eq010} and \eqref{eq011} allow us to
write
\begin{equation} \label{eq012}
    \flowB{t}{\vec{x}_0}
    = \exp{t\mat{A}} \sqb{ \vec{x}_0 + \integral{0}{\infty}{ \exp{-s\mat{A}} \vecf{b}{ \flowB{s}{\vec{x}_0} } }{s} } - \integral{t}{\infty}{ \exp{\rb{t-s}\mat{A}} \vecf{b}{ \flowB{s}{\vec{x}_0} } }{s}.
\end{equation}
See, for example, the proof of Theorem~4.1 (Stable Manifold Theorem)
of Chapter~13 of \cite{CoddingtonLevinson} which involves the trick
of ``flipping an integral'' in an integral equation. Now, there is a
${\delta>0}$ (the one in \eqref{eq011} will work) such that we can
define the function
\begin{equation} \label{eq013}
    \vecpsif{\vec{x}_0} := \vec{x}_0 + \integral{0}{\infty}{ \exp{-s\mat{A}} \vecf{b}{\flowB{s}{\vec{x}_0}} }{s},
    \quad
    \vec{x}_0 \in \ballO{\delta}.
\end{equation}
Using \eqref{eq012}, we have thus proven the following.

\begin{claim}
There is a ${\delta>0}$ such that
\begin{equation} \label{eq014}
    \flowB{t}{\vec{x}_0} = \exp{t\mat{A}} \vecpsif{\vec{x}_0} - \integral{t}{\infty}{ \exp{\rb{t-s}\mat{A}} \vecf{b}{\flowB{s}{\vec{x}_0}} }{s}
    \mymbox{for all}
    \rb{t,\vec{x}_0} \in \Omega_\delta.
\end{equation}
\end{claim}

\begin{claim}
There are ${\delta,k>0}$ such that
${\vecpsi\in\Cr{}{\ballO{\delta},\Rn}}$ and
\begin{equation} \label{eq015}
    \norm{ \vecpsif{\vec{x}_0} - \vec{x}_0 } \leq k \norm{ \vec{x}_0 }^{\alpha}
    \mymbox{for all}
    \vec{x}_0 \in \ballO{\delta}.
\end{equation}
Moreover, $\vecpsi$ is a near-identity transformation.
\end{claim}

\begin{proof}
It follows from \eqref{eq011}, \eqref{eq013}, the continuity of the
integrand, and the fact that ${\alpha>1}$.
\end{proof}

\begin{rem}
The near-identity transformation $\vecpsi$ has many interesting
properties which we will cover in a later paper. For example,
$\vecpsif{\vec{x}_0}$ is the optimal (as ${t\to\infty}$) initial
condition for the linearized system
${\dot{\vec{y}}=\mat{A}\vec{y}}$. Moreover, $\vecpsi$ satisfies the
conjugacy condition
${\vecpsif{\flowB{t}{\vec{x}_0}}=\exp{t\mat{A}}\vecpsif{\vec{x}_0}}$
and, under modest assumptions, is as smooth as the nonlinear part
$\vec{b}$.
\end{rem}

\subsection{Iterates}

\subsubsection{Definition}

Taking inspiration from the integral equation \eqref{eq014}, we will
construct iterates $\setb{\iterateD{m}{t}{\vec{y}_0}}_{m=1}^\infty$
to approximate the flow $\flowB{t}{\vec{x}_0}$.  Assuming
$\norm{\vec{y}_0}$ is sufficiently small, take as the first iterate
\begin{subequations} \label{eq016}
\begin{equation} \label{eq016a}
    \iterateD{1}{t}{\vec{y}_0} := \exp{t\mat{A}} \vec{y}_0
\end{equation}
and define the remainder recursively via
\begin{equation} \label{eq016b}
    \iterateD{m+1}{t}{\vec{y}_0} := \exp{t\mat{A}} \vec{y}_0 - \integral{t}{\infty}{ \exp{\rb{t-s}\mat{A}} \vecf{b}{\iterateD{m}{s}{\vec{y}_0}} }{s}
    \quad \rb{m\in\Nd{}}.
\end{equation}
\end{subequations}
To connect the iterates with the flow, we take
${\vec{y}_0:=\vecpsif{\vec{x}_0}}$, where $\norm{\vec{x}_0}$ is
sufficiently small. With this choice of $\vec{y}_0$, the first
iterate is the best linear approximation (as ${t\to\infty}$) of
$\flowB{t}{\vec{x}_0}$.

\begin{rem}
The computation of the iterates
$\setb{\iterateD{m}{t}{\vec{y}_0}}_{m=1}^\infty$ does not require us
to know $\vec{y}_0$, which can be treated as a parameter, in terms
of the initial condition $\vec{x}_0$. The resulting iterates, in
conjunction with Theorem~\ref{thm001} below, tell us the form of the
asymptotic expansion for the actual solution $\flowB{t}{\vec{x}_0}$.
However, it is possible to obtain an approximation of any desired
order---sometimes even the exact expression---for $\vec{y}_0$ in
terms of $\vec{x}_0$. See \S\ref{sec002.003.003}.
\end{rem}

\subsubsection{Existence, Decay Rate, and Closeness to the Flow of the Iterates}

\begin{prop} \label{prop001}
Consider the transformation defined by \eqref{eq013} and the
iterates defined by \eqref{eq016}. Let ${\sigma>0}$ be small enough
so that ${\rb{\alpha+1}\sigma<\rb{\kappa-\alpha}\mu_1}$. Then, there
is a ${\delta>0}$ (independent of $m$) such that if
${\norm{\vec{x}_0}<\delta}$ then
$\iterateD{m}{t}{\vecpsif{\vec{x}_0}}$ exists for each
${m\in\Nd{}}$. Moreover, there is a ${k>0}$ (independent of $m$)
such that
\begin{equation} \label{eq017}
    \norm{ \iterateD{m}{t}{\vecpsif{\vec{x}_0}} } \leq k \, \exp{ \rb{ \mu_1 + \sigma } t } \norm{ \vec{x}_0 }
    \mymbox{for all}
    \rb{t,\vec{x}_0} \in \Omega_\delta, ~ m \in \Nd{}.
\end{equation}
\end{prop}

\begin{proof}
Let ${\vec{y}_0:=\vecpsif{\vec{x}_0}}$. We need to be careful to
find ${\delta,k>0}$ which work for each $m$.
\begin{itemize}\setlength{\itemsep}{-\itemsep}\setlength{\listparindent}{0pt}
    \item
        Let ${\delta_1,\ell_1>0}$ be such that
        ${\norm{\vecf{b}{\vec{x}}}\leq\ell_1\norm{\vec{x}}^{\alpha}}$
        for all ${\vec{x}\in\ballO{\delta_1}}$.  See
        Equation~\eqref{eq004}.
    \item
        Let ${\delta_2>0}$ be such that ${\norm{\flowB{t}{\vec{x}_0}}<\delta_1}$
        for all ${\rb{t,\vec{x}_0}\in\Omega_{\delta_2}}$. See
        Equation~\eqref{eq006}. It follows that if
        ${\norm{\vec{x}_0}<\delta_2}$ then $\vec{y}_0$ exists.
    \item
        Let ${\ell_3>0}$ be such that
        ${\norm{\vec{y}_0}\leq\ell_3\norm{\vec{x}_0}}$ for all
        ${\vec{x}_0\in\ballO{\delta_2}}$. See
        Equation~\eqref{eq015}.
    \item
        Let ${\ell_4>0}$ be such that ${\norm{\exp{t\mat{A}}}\leq\ell_4\,\exp{\rb{\mu_1+\sigma}t}}$ for all ${t \geq
        0}$. See Equation~\eqref{eq005}.
    \item
        Let ${\ell_5>0}$ be such that ${\norm{\exp{-t\mat{A}}}\leq\ell_5\,\exp{\rb{-\kappa\mu_1+\sigma}t}}$ for all ${t \geq
        0}$. See Equation~\eqref{eq005}.
\end{itemize}
There is no harm in taking $\delta_1$ small enough so that
\[
    0 < \delta_1 < \sqb{ \frac{ \rb{\kappa\!-\!\alpha}\mu_1\!-\!\rb{\alpha\!+\!1}\sigma }{ \ell_1 \ell_5 } }^{\frac{1}{\alpha-1}}.
\]
Note that the restriction on $\sigma$ implies
\begin{equation} \label{eq018}
    \mu_1 < \mu_1 + \sigma < 0
    \mymbox{and}
    \rb{ \kappa - \alpha } \mu_1 - \rb{ \alpha + 1 } \sigma > 0.
\end{equation}
Take
\[
    k := \ell_3 \ell_4 \sqb{ 1 - \frac{ \ell_1 \ell_5 \delta_1^{\alpha-1} }{ \rb{\kappa\!-\!\alpha}\mu_1\!-\!\rb{\alpha\!+\!1}\sigma } }^{-1} > \ell_3 \ell_4
    \mymbox{and}
    \delta \in \rb{ 0, \min{}{ \frac{\delta_1}{k}, \delta_2 } }.
\]
The upper bound imposed on $\delta_1$ ensures $k$ is defined and
greater than ${\ell_3\ell_4}$. Observe
\begin{equation} \label{eq019}
    \ell_3 \ell_4 + \sqb{ \frac{ \ell_1 \ell_5 k \delta_1^{\alpha-1} }{ \rb{\kappa\!-\!\alpha}\mu_1\!-\!\rb{\alpha\!+\!1}\sigma } } = k.
\end{equation}
The proof of the result will be by induction on $m$.  Assume that
${\rb{t,\vec{x}_0}\in\Omega_\delta}$.

Consider first the base case ${m=1}$.  Since $\vec{y}_0$ exists,
so does ${\iterateD{1}{t}{\vec{y}_0}=\exp{t\mat{A}}\vec{y}_0}$. Furthermore,
\[
    \norm{ \iterateD{1}{t}{\vec{y}_0} }
    \leq \ell_4 \, \exp{\rb{\mu_1+\sigma}t} \norm{\vec{y}_0}
    \leq \ell_3 \ell_4 \, \exp{\rb{\mu_1+\sigma}t} \norm{\vec{x}_0}
    \leq k \, \exp{\rb{\mu_1+\sigma}t} \norm{\vec{x}_0}.
\]
Thus, the result is true for ${m=1}$.

Now, assume the result is true for ${m\in\Nd{}}$.  That is,
$\iterateD{m}{t}{\vec{y}_0}$ exists for
${\vec{x}_0\in\ballO{\delta}}$ and satisfies
${\norm{\iterateD{m}{t}{\vec{y}_0}} \leq k \,
\exp{\rb{\mu_1+\sigma}t} \norm{\vec{x}_0}}$. Since
${\norm{\vec{x}_0}<\delta}$, we see
${\norm{\iterateD{m}{t}{\vec{y}_0}}<\delta_1}$ and thus
$\iterateD{m+1}{t}{\vec{y}_0}$ exists. Using \eqref{eq016b}, the
induction hypothesis, and the estimates given at the beginning of
the proof,
\begin{align*}
    \norm{\iterateD{m+1}{t}{\vec{y}_0}}
    &\leq \ell_3 \ell_4 \, \exp{\rb{\mu_1+\sigma}t} \norm{\vec{x}_0} + \ell_1 \ell_5 k^{\alpha} \exp{\rb{\kappa\mu_1-\sigma}t} \sqb{ \integral{t}{\infty}{ \exp{\sqb{\rb{\alpha-\kappa}\mu_1+\rb{\alpha+1}\sigma}s} }{s} } \norm{\vec{x}_0}^{\alpha} \\
    &= \setb{ \ell_3 \ell_4 + \sqb{ \frac{ \ell_1 \ell_5 k^{\alpha} }{ \rb{\kappa\!-\!\alpha}\mu_1\!-\!\rb{\alpha\!+\!1}\sigma } } \exp{\rb{\alpha-1}\rb{\mu_1+\sigma}t} \norm{\vec{x}_0}^{\alpha-1} } \exp{\rb{\mu_1+\sigma}t} \norm{\vec{x}_0}.
\end{align*}
Since ${0 < \exp{\rb{\alpha-1}\rb{\mu_1+\sigma}t} \leq 1}$ and
${\norm{\vec{x}_0}<\delta<\fracslash{\delta_1}{k}}$, we have
\[
    \norm{\iterateD{m+1}{t}{\vec{y}_0}}
    \leq \setb{ \ell_3 \ell_4 + \sqb{ \frac{ \ell_1 \ell_5 k \delta_1^{\alpha-1} }{ \rb{\kappa\!-\!\alpha}\mu_1\!-\!\rb{\alpha\!+\!1}\sigma } } } \exp{\rb{\mu_1+\sigma}t} \norm{\vec{x}_0}.
\]
Using \eqref{eq019}, ${\norm{\iterateD{m+1}{t}{\vec{y}_0}} \leq
k\,\exp{\rb{\mu_1+\sigma}t}\norm{\vec{x}_0}}$. Hence, the result is
true for ${m+1}$. By induction, the result is true for each
${m\in\Nd{}}$.
\end{proof}

\begin{thm} \label{thm001}
Consider the transformation defined by \eqref{eq013} and the
iterates defined by \eqref{eq016}. Let ${\sigma>0}$ be small enough
so that ${\rb{\alpha+1}\sigma<\rb{\kappa-\alpha}\mu_1}$. Then, there
are constants ${\delta>0}$ and
${\seq{k}{m}{1}{\infty}\subset\Rd{}_+}$ such that, for all
${\rb{t,\vec{x}_0}\in\Omega_\delta}$ and ${m\in\Nd{}}$,
\begin{equation} \label{eq020}
    \norm{ \flowB{t}{\vec{x}_0} - \iterateD{m}{t}{\vecpsif{\vec{x}_0}} }
    \leq k_m \exp{ \sqb{ \alpha + \rb{ m - 1 } \beta } \sqb{ \mu_1 + \sigma  } t } \norm{ \vec{x}_0 }^{ \alpha + \rb{ m - 1 } \beta }.
\end{equation}
\end{thm}

\begin{proof}
Let ${\vec{y}_0:=\vecpsif{\vec{x}_0}}$. We need to be careful to
find ${\delta>0}$ which works for each ${m\in\Nd{}}$.
\begin{itemize}\setlength{\itemsep}{-\itemsep}
    \item
        Let ${\delta_1,\ell_1>0}$ be such that
        ${\norm{\vecf{b}{\vec{x}}}\leq\ell_1\norm{\vec{x}}^{\alpha}}$
        for all ${\vec{x}\in\ballO{\delta_1}}$.  See
        Equation~\eqref{eq004}.
    \item
        Let ${\delta_2,\ell_2>0}$ be such that ${\norm{\mat{D}\vecf{b}{\vec{x}}}\leq\ell_2\norm{\vec{x}}^{\beta}}$
        for all ${\vec{x}\in\ballO{\delta_2}}$.  See
        Equation~\eqref{eq004}.
    \item
        Let ${\delta_3,\ell_3>0}$ be such that
        ${\norm{\flowB{t}{\vec{x}_0}}\leq\ell_3\,\exp{\rb{\mu_1+\sigma}t}\norm{\vec{x}_0}}$
        for all ${\rb{t,\vec{x}_0}\in\Omega_{\delta_3}}$. See
        Equation~\eqref{eq006}.
    \item
        Let ${\delta_4,\ell_4>0}$ be such that
        ${\norm{\iterateD{m}{t}{\vec{y}_0}}\leq\ell_4\,\exp{\rb{\mu_1+\sigma}t}\norm{\vec{x}_0}}$
        for all ${\rb{t,\vec{x}_0}\in\Omega_{\delta_4}}$ and ${m\in\Nd{}}$. See
        Equation~\eqref{eq017}.
    \item
        Let ${\ell_5>0}$ be such that ${\norm{\exp{-t\mat{A}}}\leq\ell_5\,\exp{(-\kappa\mu_1+\sigma)t}}$ for all ${t \geq
        0}$. See Equation~\eqref{eq005}.
    \item
        If ${\norm{\vec{x}_0}<\min{}{\delta_3,\delta_4}}$, then any
        point on the line segment connecting $\flowB{t}{\vec{x}_0}$ and
        $\iterateD{m}{t}{\vec{y}_0}$ is bounded in norm by ${\rb{\ell_3+\ell_4}\exp{\rb{\mu_1+\sigma}t}\norm{\vec{x}_0}}$ for
        all ${t \geq 0}$.
\end{itemize}
The proof of the theorem will be by induction on $m$.  Take
\[
    \delta := \min{}{\frac{\delta_1}{\ell_3},\frac{\delta_1}{\ell_4},\frac{\delta_2}{\ell_3+\ell_4}, \delta_3, \delta_4 }.
\]
Hence assume ${\rb{t,\vec{x}_0}\in\Omega_\delta}$ so all necessary estimates apply.

Consider first the base case, ${m=1}$. Using \eqref{eq014} and
\eqref{eq016a} along with the estimates at the beginning of the
proof,
\[
    \norm{ \flowB{t}{\vec{x}_0} - \iterateD{1}{t}{\vec{y}_0} }
    \leq \integral{t}{\infty}{ \norm{ \exp{\rb{t-s}\mat{A}} } \norm{ \vecf{b}{ \flowB{s}{\vec{x}_0} } } }{s}
    \leq \underbrace{ \sqb{ \frac{ \ell_1 \ell_3^{\alpha} \ell_5 }{ \rb{\kappa\!-\!\alpha}\mu_1\!-\!\rb{\alpha\!+\!1}\sigma } } }_{k_1} \exp{ \alpha \rb{ \mu_1 + \sigma } t } \norm{\vec{x}_0}^{\alpha}.
\]
Using \eqref{eq018}, which is valid in this proof as well,
${k_1>0}$. Thus, the theorem is true for ${m=1}$.

Now, assume the theorem is true for fixed ${m\in\Nd{}}$.
First, for any ${s \geq 0}$ consider the expression
\[
    \norm{ \vecf{b}{ \flowB{s}{\vec{x}_0} } - \vecf{b}{ \iterateD{m}{s}{\vec{y}_0} } }
    \leq \ell_2 \rb{\ell_3+\ell_4}^{\beta} \exp{\beta\rb{\mu_1+\sigma}s} \norm{ \flowB{s}{\vec{x}_0} - \iterateD{m}{s}{\vec{y}_0} } \norm{\vec{x}_0}^{\beta},
\]
where we applied the Mean Value Theorem and estimates given at the
beginning of the proof.  By the induction hypothesis, we know
\begin{align*}
    \norm{ \vecf{b}{ \flowB{s}{\vec{x}_0} } - \vecf{b}{ \iterateD{m}{s}{\vec{y}_0} } }
    &\leq \ell_2 \rb{\ell_3+\ell_4}^{\beta} \exp{\beta\rb{\mu_1+\sigma}s} \setb{ k_m \exp{ \sqb{ \alpha + \rb{ m - 1 } \beta } \sqb{ \mu_1 + \sigma  } s } \norm{ \vec{x}_0 }^{ \alpha + \rb{ m - 1 } \beta } } \norm{\vec{x}_0}^{\beta} \\
    &= k_m \ell_2 \rb{\ell_3+\ell_4}^{\beta} \exp{\rb{\alpha+m\beta}\rb{\mu_1+\sigma}s} \norm{\vec{x}_0}^{\alpha+m\beta}.
\end{align*}
It follows from the integral equation \eqref{eq014} for
$\flowB{t}{\vec{x}_0}$ and the definition \eqref{eq016b} for
$\iterateD{m+1}{s}{\vec{y}_0}$ that
\begin{align*}
    \norm{ \flowB{t}{\vec{x}_0} - \iterateD{m+1}{t}{\vec{y}_0} }
    &\leq \integral{t}{\infty}{ \norm{ \exp{\rb{t-s}\mat{A}} } \norm{ \vecf{b}{ \flowB{s}{\vec{x}_0} } - \vecf{b}{ \iterateD{m}{s}{\vec{y}_0} } } }{s} \\
    &\leq k_m \ell_2 \rb{\ell_3+\ell_4}^{\beta} \ell_5 \setb{ \integral{t}{\infty}{ \exp{\rb{\kappa\mu_1-\sigma}\rb{t-s}} \exp{\rb{\alpha+m\beta}\rb{\mu_1+\sigma}s} }{s} } \norm{\vec{x}_0}^{\alpha+m\beta} \\
    &= \underbrace{ \sqb{ \frac{ k_m \ell_2 \rb{\ell_3+\ell_4}^{\beta} \ell_5 }{ \rb{\kappa-\alpha}\mu_1 - \rb{\alpha+1}\sigma-m\beta\rb{\mu_1+\sigma} } } }_{k_{m+1}} \exp{\rb{\alpha+m\beta}\rb{\mu_1+\sigma}t} \norm{\vec{x}_0}^{\alpha+m\beta}.
\end{align*}
Using \eqref{eq018}, ${k_{m+1}>0}$. Thus, the theorem is true for
${m+1}$ and, by induction, for all ${m\in\Nd{}}$.
\end{proof}

\begin{rem} \label{rem001}
Consider the constants ${\seq{k}{m}{1}{\infty}\subset\Rd{}_+}$ from
Theorem~\ref{thm001} and its proof. It is apparent from their
construction there is some constant ${r>0}$, independent of $m$,
such that ${k_{m+1}\leq\rb{\fracslash{r}{m}}k_m}$ for all
${m\in\Nd{}}$ and hence ${k_{m+1}\leq\rb{\fracslash{r^m}{m!}}k_1}$
for all ${m\in\Nd{}}$. Consequently, ${k_m \to 0}$ as
${m\to\infty}$.
\end{rem}

\begin{cor} \label{cor001}
Consider the transformation defined by \eqref{eq013} and the
iterates defined by \eqref{eq016}. There exists ${\delta>0}$ such
that
\dsm{\lim_{m\to\infty}\iterateD{m}{t}{\vecpsif{\vec{x}_0}}=\flowB{t}{\vec{x}_0}}
uniformly in $\Omega_\delta$.
\end{cor}

\begin{proof}
It follows from Theorem~\ref{thm001} and Remark~\ref{rem001}.
\end{proof}

\subsection{Approximations}

\subsubsection{Truncating the Iterates}

Theorem~\ref{thm001} guarantees that $\iterateD{m}{t}{\vec{y}_0}$
correctly approximates $\flowB{t}{\vec{x}_0}$ up to a certain order
when ${\vec{y}_0=\vecpsif{\vec{x}_0}}$. However,
$\iterateD{m}{t}{\vec{y}_0}$ may contain higher-order terms.
Fortunately, as is evident from the proof of the theorem, we can
discard the terms of $\iterateD{m}{t}{\vec{y}_0}$ which are not
guaranteed to be correct by the theorem. If we use this simplified
iterate in the calculation of $\iterateD{m+1}{t}{\vec{y}_0}$ we will
not sacrifice accuracy. This potentially can simplify the
calculation of iterates greatly.

\subsubsection{Truncating the Nonlinear Part}

The integration in the recursive definition \eqref{eq016b} of the
iterates may be difficult if not impossible. However, it is possible
to replace $\vecf{b}{\vec{x}}$ with a Taylor polynomial of an
appropriate order and not sacrifice accuracy. Assume that
${\vec{b}\in\Cr{\infty}{\ballO{\delta},\Rn}}$, where ${\delta>0}$ is
small.  By Taylor's Theorem, we can write
${\vecf{b}{\vec{x}}\sim\sum_{i=\ell}^\infty \vecif{b}{i}{\vec{x}}}$
as ${\norm{\vec{x}} \to 0}$ for some index
${\ell\in\setb{2,3,\ldots}}$, where each $\vecif{b}{i}{\vec{x}}$ has
components which are homogeneous polynomials of degree $i$. Note
that we can take ${\alpha=\ell}$ and ${\beta=\ell-1}$. Define
${\vecmf{b}{m}{\vec{x}} := \sum_{i=\ell}^m \vecif{b}{i}{\vec{x}}}$
for each ${m\in\setb{\ell,\ell+1,\ldots}}$, which is the Taylor
polynomial of $\vecf{b}{\vec{x}}$ of order $m$. Assuming
$\norm{\vec{x}_0}$ is sufficiently small and
${\vec{y}_0=\vecpsif{\vec{x}_0}}$, we can define
$\iterateD{m+1}{t}{\vec{y}_0}$ in exactly the same way as before
except with $\vecf{b}{\vec{x}}$ replaced by
$\vecmf{b}{\sqb{m+1}\sqb{\ell-1}}{\vec{x}}$. The resulting closeness
of $\iterateD{m}{t}{\vec{y}_0}$ to $\flowB{t}{\vec{x}_0}$ is no
different than that stated in Theorem~\ref{thm001}.

\subsubsection{Approximating the Transformation} \label{sec002.003.003}

Consider the pivotal transformation $\vecpsi$, defined in
\eqref{eq013}. The iterates \eqref{eq016} can be computed in terms
of $\vec{y}_0$ without knowing $\vecpsif{\vec{x}_0}$ explicitly.
Moreover, the iterates give the form of the asymptotic expansion for
$\flowB{t}{\vec{x}_0}$ with error estimates courtesy of
Theorem~\ref{thm001}. However, we can approximate
$\vecpsif{\vec{x}_0}$ if necessary with an iterative method. In a
later paper, we present an alternative method for finding or
approximating $\vecpsif{\vec{x}_0}$.

We will construct a sequence of approximations
$\setb{\vecpsimf{m}{\vec{x}_0}}_{m=1}^\infty$ for
$\vecpsif{\vec{x}_0}$. Assuming $\norm{\vec{x}_0}$ is sufficiently
small, Equations~\eqref{eq013} and \eqref{eq015} along with
Theorem~\ref{thm001} suggest we take
\begin{equation} \label{eq021}
    \vecpsimf{1}{\vec{x}_0} := \vec{x}_0
    \mymbox{and}
    \vecpsimf{m+1}{\vec{x}_0} := \vec{x}_0 + \integral{0}{\infty}{ \exp{-s\mat{A}} \vecf{b}{ \iterateD{m}{s}{\vecpsimf{m}{\vec{x}_0}} } }{s}
    \quad
    \rb{ m \in \Nd{} }.
\end{equation}
We can state how close $\vecpsimf{m}{\vec{x}_0}$ is to
$\vecpsif{\vec{x}_0}$. Claim~\ref{claim002}, required to prove
Proposition~\ref{prop002}, can be combined with Theorem~\ref{thm001}
and Proposition~\ref{prop002} to state how close
$\iterateD{m}{t}{\vecpsimf{m}{\vec{x}_0}}$ is to
$\flowB{t}{\vec{x}_0}$.

\begin{claim} \label{claim002}
Consider the transformation defined by \eqref{eq013} and the
iterates defined by \eqref{eq016}. For any sufficiently small
${\sigma>0}$, there exist ${\delta,k>0}$ such that, for all
${\vec{x}_0,\vec{y}_0\in\ballO{\delta}}$, ${t \geq 0}$, and
${m\in\Nd{}}$,
\[
    \norm{ \iterateD{m}{t}{\vecpsif{\vec{x}_0}} - \iterateD{m}{t}{\vec{y}_0} }
    \leq k \, \exp{ \rb{ \mu_1 + \sigma } t } \norm{ \vecpsif{\vec{x}_0} - \vec{y}_0 }.
\]
\end{claim}

\begin{proof}
The result follows intuitively from \eqref{eq016} and the proof is
omitted in the interest of space.
\end{proof}

\begin{prop} \label{prop002}
Consider the transformation defined by \eqref{eq013} and the
approximations defined by \eqref{eq021}. There are constants
${\delta>0}$ and ${\seq{k}{m}{1}{\infty}\subset\Rd{}_+}$ such that
\[
    \norm{ \vecpsif{\vec{x}_0} - \vecpsimf{m}{\vec{x}_0} }
    \leq k_m \norm{ \vec{x}_0 }^{\alpha+\rb{m-1}\beta}
    \mymbox{for all}
    \vec{x}_0 \in \ballO{\delta}
\]
and \dsm{\lim_{m\to\infty}\vecpsimf{m}{\vec{x}_0}=\vecpsif{\vec{x}_0}} uniformly in $\ballO{\delta}$.
\end{prop}

\begin{proof}
The proof is omitted in the interest of space. It involves
induction, the Triangle Inequality, the Mean Value Theorem,
Theorem~\ref{thm001}, and Claim~\ref{claim002}.
\end{proof}

\subsection{Examples}

\subsubsection{Star Node with Quadratic Nonlinearity} \label{sec002.004.001}

Take
\[
    \mat{A} := \twobytwomatrix{a}{0}{0}{a}
    \mymbox{and}
    \vecf{b}{\vec{x}} := \twobyonematrix{ b_{111} x_1^2 + b_{112} x_1 x_2 + b_{122} x_2^2 }{ b_{211} x_1^2 + b_{212} x_1 x_2 + b_{222} x_2^2 },
\]
where ${a<0}$ and ${\setb{b_{ijk}}_{i,j,k=1}^2\subset\Rd{}}$ are
constants.  Assume that at least one $b_{ijk}$ is non-zero and
consider the system
${\dot{\vec{x}}=\mat{A}\vec{x}+\vecf{b}{\vec{x}}}$. The origin, as we
see, is a star node. Now, ${\exp{t\mat{A}}=\exp{at}\mat{I}}$,
${\mu_1=a=\mu_2}$, and ${\kappa=1}$. Also, we can take ${\alpha=2}$
and ${\beta=1}$.  Thus, ${\kappa<\alpha}$ and the eigenvalues are
closely-spaced relative to the nonlinear part.

Define ${\vec{y}_0:=\vecpsif{\vec{x}_0}}$ and
${\vec{x}(t):=\flowB{t}{\vec{x}_0}}$, where $\norm{\vec{x}_0}$ is
sufficiently small. The first iterate, which gives us the
linearization, is ${\iterateD{1}{t}{\vec{y}_0}=\exp{at}\vec{y}_0}$.
Theorem~\ref{thm001} tells us ${x_1(t)\sim\exp{at}y_{01}}$ and
${x_2(t)\sim\exp{at}y_{02}}$ as ${t\to\infty}$. Provided ${y_{01}
\ne 0}$, ${x_2(t)\sim\rb{\fracslash{y_{02}}{y_{01}}}x_1(t)}$ as
${t\to\infty}$. The Hartman-Grobman Theorem yields basically the
same conclusion. We will use the second iterate to obtain a more
detailed relationship between $x_1(t)$ and $x_2(t)$. Using
\eqref{eq016b} and the fact
${\vecf{b}{\exp{at}\vec{y}_0}=\exp{2at}\vecf{b}{\vec{y}_0}}$,
\begin{equation} \label{eq022}
    \iterateD{2}{t}{\vec{y}_0}
    = \exp{at} \vec{y}_0 + \exp{2at} \vecxi,
    \mymbox{where}
    \vecxi := \rb{ \frac{1}{a} } \vecf{b}{\vec{y}_0} .
\end{equation}

\begin{description}
    \item[$\diamond$ Case 1: ${y_{01}=y_{02}=0}$.]
        If ${\vec{y}_0=\vec{0}}$, ${\iterateD{m}{t}{\vec{y}_0}=\vec{0}}$ for each ${m\in\Nd{}}$. By Corollary~\ref{cor001}, ${\vec{x}(t) \equiv \vec{0}}$.
    \item[$\diamond$ Case 2: ${y_{01} \ne 0}$.]
        It follows from \eqref{eq022} and Theorem~\ref{thm001} that, as ${t\to\infty}$,
        \begin{equation} \label{eq023}
            x_1(t) = \exp{at} y_{01} + \exp{2at} \xi_1 + \bigO{\exp{3at}}
            \mymbox{and}
            x_2(t) = \exp{at} y_{02} + \exp{2at} \xi_2 + \bigO{\exp{3at}}.
        \end{equation}
        Using \eqref{eq023}, ${x_1=\bigO{\exp{at}}}$ as ${t\to\infty}$ and ${\exp{at}=\bigO{x_1}}$ as ${x_1 \to 0}$. Now, the first equation of \eqref{eq023} can be written, as ${t\to\infty}$,
        \[
            \exp{at}
            = \rb{ \frac{x_1}{y_{01}} } \sqb{ 1 - \rb{ \frac{\xi_1}{y_{01}} } \exp{at} + \bigO{\exp{2at}} }
            = \rb{ \frac{1}{y_{01}} } x_1 - \rb{ \frac{\xi_1}{y_{01}^2} } \exp{at} x_1 + \bigO{\exp{3at}},
        \]
        where we used the fact ${\sqb{1+u+\bigO{u^2}}^{-1}=1-u+\bigO{u^2}}$ as ${u \to 0}$. Back-substituting $\exp{at}$,
        \[
            \exp{at}
            = \rb{ \frac{1}{y_{01}} } x_1 - \rb{ \frac{\xi_1}{y_{01}^3} } x_1^2 + \bigO{\exp{3at}}
            = \rb{ \frac{1}{y_{01}} } x_1 - \rb{ \frac{\xi_1}{y_{01}^3} } x_1^2 + \bigO{x_1^3}
        \]
        as ${t\to\infty}$ for the former and ${x_1 \to 0}$ for the latter. From the second equation of \eqref{eq023},
        \begin{equation} \label{eq024}
            x_2 = \rb{ \frac{y_{02}}{y_{01}} } x_1 + \rb{ \frac{\xi_2}{y_{01}^2} - \frac{y_{02}\xi_1}{y_{01}^3} } x_1^2 + \bigO{x_1^3}
            \mymbox{as}
            x_1 \to 0.
        \end{equation}
    \item[$\diamond$ Case 3: ${y_{02} \ne 0}$.]
        Similar to the case above, we get
        \[
            x_1 = \rb{ \frac{y_{01}}{y_{02}} } x_2 + \rb{ \frac{\xi_1}{y_{02}^2} - \frac{y_{01}\xi_2}{y_{02}^3} } x_2^2 + \bigO{x_2^3}
            \mymbox{as}
            x_2 \to 0.
        \]
\end{description}

For a more concrete example, consider the system
\begin{equation} \label{eq025}
    \twobyonematrix{\dot{x}_1}{\dot{x}_2}
    = \twobytwomatrix{-1}{0}{0}{-1} \twobyonematrix{x_1}{x_2} + \twobyonematrix{ x_1^2 + 8 \, x_1 x_2 + x_2^2 }{ 8 \, x_1^2 + x_1 x_2 + 8 \, x_2^2 }.
\end{equation}
Using \eqref{eq024}, provided ${y_{01} \ne 0}$,
\begin{equation} \label{eq026}
    x_2 = \rb{ \frac{y_{02}}{y_{01}} } x_1 + \sqb{ \rb{ \frac{y_{02}}{y_{01}} }^3 - 8 } x_1^2 + \bigO{x_1^3}
    \mymbox{as}
    x_1 \to 0.
\end{equation}
We can use the sign of ${\rb{\fracslash{y_{02}}{y_{01}}}^3-8}$ in
\eqref{eq026} to deduce concavity. A phase portrait for
\eqref{eq025} is sketched in Figure~\ref{fig001}.

\begin{figure}[t]
\begin{center}
    \includegraphics[width=400pt]{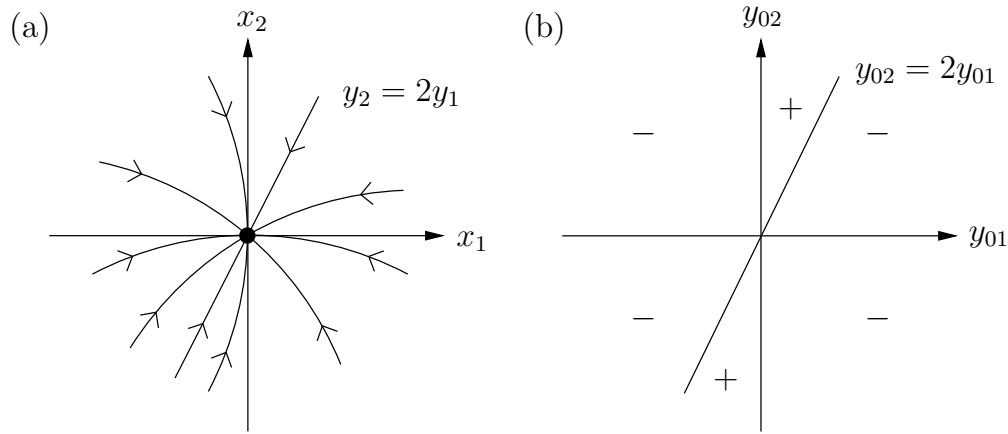}
    \caption[Phase portrait near the origin for an example system]{(a) The phase portrait
    near the origin for the example system \eqref{eq025}. (b) The sign of
    ${\rb{\fracslash{y_{02}}{y_{01}}}^3-8}$ in
    different regions of the $y_{01}y_{02}$-plane, using \eqref{eq015} and \eqref{eq026}, specifies concavity.} \label{fig001}
\end{center}
\end{figure}

\subsubsection{Example of Finding the Transformation}

Let
\[
    \mat{A} := \twobytwomatrix{-1}{0}{0}{-2}
    \mymbox{and}
    \vecf{b}{\vec{x}} := \twobyonematrix{0}{x_1^3}
\]
and consider the system
${\dot{\vec{x}}=\mat{A}\vec{x}+\vecf{b}{\vec{x}}}$. We will use the
approximations \eqref{eq021} to find $\vecpsif{\vec{x}_0}$. Here,
\[
    \vecpsimf{1}{\vec{x}_0} = \vec{x}_0
    \mymbox{and}
    \vecpsimf{m+1}{\vec{x}_0}
    = \twobyonematrix{ x_{01} }{ x_{02} + \integral{0}{\infty}{ \exp{2s} \sqb{ \iterateDi{1}{m}{s}{ \vecpsimf{m}{\vec{x}_0} } }^3 }{s} }
    \mymbox{for}
    m \in \Nd{}.
\]
An easy inductive argument and Proposition~\ref{prop002} establish
\[
    \iterateDi{1}{m}{t}{\vec{y}_0} = \exp{-t} y_{01},
    \quad
    \vecpsimf{m+1}{\vec{x}_0} = \twobyonematrix{ x_{01} }{ x_{02} + x_{01}^3 },
    \mymbox{and}
    \vecpsif{\vec{x}_0} = \twobyonematrix{ x_{01} }{ x_{02} + x_{01}^3 }.
\]

\subsection{Transforming Widely-Spaced to Closely-Spaced Eigenvalues}

This section deals with the case where the eigenvalues are
closely-spaced relative to the nonlinear part of the differential
equation.  That is, ${\kappa<\alpha}$.  Momentarily, we will deal with the
case where the eigenvalues are widely-spaced relative to the
nonlinear part.  That is, ${\kappa\geq\alpha}$. However, the techniques for widely-spaced eigenvalues are
decidedly more tedious. Fortunately, in practice it is sometimes
possible to transform a system with widely-spaced eigenvalues to a
system with closely-spaced eigenvalues. The transformation need not
be polynomial or even analytic.

The reduction of widely-spaced eigenvalues to closely-spaced eigenvalues is related to the theory of
normal forms. Assume the nonlinear vector field $\vec{b}$ is
analytic (or sufficiently smooth) at the origin and the eigenvalues
of the matrix $\mat{A}$ are widely-spaced relative to $\vec{b}$.

Suppose that there is no resonance in the eigenvalues of $\mat{A}$.
For a given ${m\in\Nd{}}$ with ${m>\kappa}$, by the Normal Form Theorem there exists a
sequence of near-identity, analytic transformations such that the
differential equation
${\dot{\vec{x}}=\mat{A}\vec{x}+\vecf{b}{\vec{x}}}$ is converted to a
new differential equation of the form
${\dot{\vec{y}}=\mat{A}\vec{y}+\vecf{c}{\vec{y}}}$ where
${\norm{\vecf{c}{\vec{y}}}=\bigO{\norm{\vec{y}}^m}}$ as
${\norm{\vec{y}} \to 0}$. The new system, of course, has eigenvalues
which are closely-spaced relative to the nonlinear part.

Suppose now that there is resonance in the eigenvalues of $\mat{A}$,
say with lowest order ${m\in\Nd{}}$. The best we can do is find a sequence of
near-identity, analytic transformations to convert
${\dot{\vec{x}}=\mat{A}\vec{x}+\vecf{b}{\vec{x}}}$ to
${\dot{\vec{y}}=\mat{A}\vec{y}+\vecf{c}{\vec{y}}}$ where
${\norm{\vecf{c}{\vec{y}}}=\bigO{\norm{\vec{y}}^m}}$ as
${\norm{\vec{y}} \to 0}$. If ${m>\kappa}$, then the eigenvalues of
the new system will be closely-spaced relative to the nonlinear
part.

\section{Widely-Spaced Eigenvalues}
\label{sec003}

\subsection{Introduction}

For simplicity, we will restrict our attention to a special case
which is most applicable to us. The general case is outlined in
\S\ref{sec003.005}. Consider the initial value problem
\begin{equation} \label{eq027}
    \dot{\vec{x}} = \mat{A} \vec{x} + \vecf{b}{\vec{x}},
    \quad
    \vec{x}(0) = \vec{x}_0,
    \mymbox{where}
    \mat{A} := \twobytwomatrix{a}{0}{0}{\kappa a}
    \mymbox{and}
    \vecf{b}{\vec{x}} = \twobyonematrix{ \scif{b}{1}{\vec{x}} }{ \scif{b}{2}{\vec{x}} }
\end{equation}
with ${a<0}$ and ${\kappa \geq 1}$. The matrix $\mat{A}$ and vector
field $\vec{b}$ satisfy the same assumptions as before (see
\eqref{eq003} and \eqref{eq004}). Additionally,
we assume
\begin{equation} \label{eq028}
    \kappa \geq \alpha.
\end{equation}
That is, the eigenvalues are widely-spaced relative to the nonlinear
part. Note that ${\kappa a \leq \alpha a < a < 0}$. Note also that
\eqref{eq027} can be written, where
${\vec{x}_0=\rb{x_{01},x_{02}}^T}$, as the integral equation
\begin{equation} \label{eq029}
    \flowC{t}{\vec{x}_0}
    = \twobyonematrix{ \exp{a t} x_{01} + \integral{0}{t}{ \exp{a\rb{t-s}} \scif{b}{1}{\flowC{s}{\vec{x}_0}} }{s} }{ \exp{\kappa a t} x_{02} + \integral{0}{t}{ \exp{\kappa a\rb{t-s}} \scif{b}{2}{\flowC{s}{\vec{x}_0}} }{s} }.
\end{equation}

We can specialize the basic estimates to the case at hand. First, we
know from Claim~\ref{claim001} that
\begin{equation} \label{eq030}
    \norm{ \exp{t\mat{A}} } = \exp{at}
    \mymbox{and}
    \norm{ \exp{-t\mat{A}} } = \exp{-\kappa a t}
    \mymbox{for all} t \geq 0.
\end{equation}
Furthermore, since $\mat{A}$ is diagonal, we know from \eqref{eq006}
that there are ${\delta,k>0}$ such that
\begin{equation} \label{eq031}
    \norm{ \flowC{t}{\vec{x}_0} } \leq k \, \exp{at} \norm{\vec{x}_0}
    \mymbox{for all}
    \rb{t,\vec{x}_0} \in \Omega_\delta.
\end{equation}

\begin{claim}
For any ${\sigma>0}$, there are ${\delta,k>0}$ such that
\begin{equation} \label{eq032}
    \abs{ \flowCi{2}{t}{\vec{x}_0} }
    \leq k \, \exp{ \rb{ \alpha a + \sigma } t } \norm{\vec{x}_0}
    \mymbox{for all}
    \rb{t,\vec{x}_0} \in \Omega_\delta.
\end{equation}
\end{claim}

\begin{proof}
Applying the estimates \eqref{eq004} and \eqref{eq031} to the second
component of \eqref{eq029}, there are ${\delta,k_1,k_2>0}$ such that
if ${\rb{t,\vec{x}_0}\in\Omega_\delta}$ then
\[
    \abs{ \flowCi{2}{t}{\vec{x}_0} }
    \leq k_1 \, \exp{\kappa a t} \norm{\vec{x}_0} + k_2 \, \exp{\kappa at} \setb{ \integral{0}{t}{ \exp{\sqb{\sigma-\rb{\kappa-\alpha}a}s} }{s} } \norm{\vec{x}_0}^\alpha,
\]
where we inserted arbitrary ${\sigma>0}$ to allow for the
possibility that ${\kappa=\alpha}$. Simplify to get \eqref{eq032}.
\end{proof}

\begin{rem}
When ${\kappa>\alpha}$, we can take ${\sigma=0}$ in \eqref{eq032}
and throughout this section except for \S\ref{sec003.005}.
\end{rem}

\subsection{Iterates}

\subsubsection{First Component of the Transformation}

In the previous section, where the eigenvalues were closely-spaced
relative to the nonlinear part, we were able to ``flip the
integral'' in the general integral equation \eqref{eq002}. However,
we cannot do this in this section. Why? The estimates \eqref{eq004},
\eqref{eq030}, and \eqref{eq031} tell us there are ${\delta,k>0}$
such that
\[
    \norm{ \exp{-t\mat{A}} \vecf{b}{\flowC{t}{\vec{x}_0}} }
    \leq k \, \exp{\rb{\alpha-\kappa}at} \norm{\vec{x}_0}^\alpha
    \mymbox{for all}
    \rb{t,\vec{x}_0} \in \Omega_\delta.
\]
The condition \eqref{eq028} prevents us from concluding that
$\norm{\exp{-t\mat{A}}\vecf{b}{\flowC{t}{\vec{x}_0}}}$ decays
exponentially.  This in turn means that we cannot, in general,
``flip the integral'' in \eqref{eq002}. Hence, we cannot find an
appropriate $\vec{y}_0$ and define the iterates as easily as with
the case of closely-spaced eigenvalues.

\begin{claim}
There are ${\delta,k>0}$ such that
\begin{equation} \label{eq033}
    \abs{ \exp{-a t} \scif{b}{1}{\flowC{t}{\vec{x}_0}} }
    \leq k \, \exp{\rb{\alpha-1}at} \norm{ \vec{x}_0 }^\alpha
    \mymbox{for all}
    \rb{t,\vec{x}_0} \in \Omega_\delta.
\end{equation}
\end{claim}

\begin{proof}
It follows from \eqref{eq004} and \eqref{eq031}.
\end{proof}

Since ${\alpha>1}$, we see from \eqref{eq033} that we can ``flip the
integral'' in the first component of \eqref{eq029}.

\begin{claim}
There is a ${\delta>0}$ such that
\begin{equation} \label{eq034}
    \flowC{t}{\vec{x}_0}
    = \twobyonematrix{ \exp{a t} \psiif{1}{\vec{x}_0} - \integral{t}{\infty}{ \exp{a\rb{t-s}} \scif{b}{1}{\flowC{s}{\vec{x}_0}} }{s} }{ \exp{\kappa a t} x_{02} + \integral{0}{t}{ \exp{\kappa a\rb{t-s}} \scif{b}{2}{\flowC{s}{\vec{x}_0}} }{s} }
\end{equation}
for all ${\rb{t,\vec{x}_0}\in\Omega_\delta}$, where
\begin{equation} \label{eq035}
    \psiif{1}{\vec{x}_0} := x_{01} + \integral{0}{\infty}{ \exp{-a s} \scif{b}{1}{\flowC{s}{\vec{x}_0}} }{s}.
\end{equation}
\end{claim}

\begin{claim}
There are ${\delta,k>0}$ such that
\begin{equation} \label{eq036}
    \abs{ \psiif{1}{\vec{x}_0} - x_{01} } \leq k \norm{\vec{x}_0}^\alpha
    \mymbox{for all}
    \vec{x}_0 \in \ballO{\delta}.
\end{equation}
\end{claim}

\begin{proof}
It follows from \eqref{eq033} and \eqref{eq035}.
\end{proof}

\subsubsection{First Group of Iterates}

The iterates, which will be denoted by
$\setb{\iterateD{m}{t}{\vec{y}_0}}_{m=1}^\infty$, must be split
into two groups. Define
\[
    p := \floor{ \frac{\kappa-\alpha}{\beta} } + 2,
\]
which is such that the $p^\text{th}$, but not $\rb{p-1}^\text{th}$,
iterate will be close enough to $\flowC{t}{\vec{x}_0}$ so that we
can perform another ``trick'' in the second component of
\eqref{eq034}. It can be shown
\begin{equation} \label{eq037}
    p \in \setb{2,3,\ldots}
    \mymbox{and}
    \alpha + \rb{p-2} \beta \leq \kappa < \alpha + \rb{p-1} \beta.
\end{equation}
If $\norm{\vec{y}_0}$ is sufficiently small, we take as the first iterate
\begin{subequations} \label{eq038}
\begin{equation} \label{eq038a}
    \iterateD{1}{t}{y_{01}} := \twobyonematrix{ \exp{a t} y_{01} }{0}
\end{equation}
and (if ${p>2}$) define recursively
\begin{equation} \label{eq038b}
    \iterateD{m+1}{t}{y_{01}} := \twobyonematrix{ \exp{a t} y_{01} - \integral{t}{\infty}{ \exp{a\rb{t-s}} \scif{b}{1}{ \iterateD{m}{s}{y_{01}} } }{s} }{ \integral{0}{t}{ \exp{\kappa a\rb{t-s}} \scif{b}{2}{ \iterateD{m}{s}{y_{01}} } }{s} }
    \quad
    \rb{ 1 \leq m \leq p-2 }.
\end{equation}
\end{subequations}
To connect these iterates with the flow, choose ${y_{01}:=\psiif{1}{\vec{x}_0}}$ with
$\norm{\vec{x}_0}$ sufficiently small.

\begin{rem}
Since the first group of iterates does not depend on $y_{02}$, we
write $\iterateD{m}{t}{y_{01}}$ instead of
$\iterateD{m}{t}{\vec{y}_0}$. This is necessary since
$\psiif{2}{\vec{x}_0}$ will be defined in terms of
$\iterateD{p-1}{t}{\psiif{1}{\vec{x}_0}}$. However, if we need to
refer to all the iterates at once we will use, for simplicity, the
notation $\setb{\iterateD{m}{t}{\vec{y}_0}}_{m=1}^\infty$.
\end{rem}

\subsubsection{Existence, Decay Rate, and Closeness to the Flow of the First Group}

\begin{prop}
Consider the transformation defined by \eqref{eq035} and the first
group of iterates defined by \eqref{eq038}. Then, there is a
${\delta>0}$ (independent of $m$) such that if
${\norm{\vec{x}_0}<\delta}$ then $\iterateD{m}{t}{y_{01}}$ exists
for each ${m\in\setb{1,\ldots,p-1}}$, where
${y_{01}:=\psiif{1}{\vec{x}_0}}$. Moreover, there is a ${k>0}$
(independent of $m$) such that
\begin{equation} \label{eq039}
    \norm{ \iterateD{m}{t}{y_{01}} } \leq k \, \exp{at} \norm{ \vec{x}_0 }
    \mymbox{for all}
    \rb{t,\vec{x}_0}\in\Omega_\delta, ~ m\in\setb{1,\ldots,p-1}.
\end{equation}
\end{prop}

\begin{proof}
The proof is similar to that of Proposition~\ref{prop001} and omitted in the interest of space.
\end{proof}

\begin{thm} \label{thm002}
Consider the transformation defined by \eqref{eq035} and the first
group of iterates defined by \eqref{eq038}. Let ${\sigma>0}$ be
sufficiently small. Then, there exist ${\delta>0}$ and
${\seq{k}{m}{1}{p-1}\subset\Rd{}_+}$ such that
\begin{equation} \label{eq040}
    \norm{ \flowC{t}{\vec{x}_0} - \iterateD{m}{t}{y_{01}} } \leq k_m \, \exp{ \sqb{ \alpha + \rb{ m - 1 } \beta } \sqb{ a + \sigma } t } \norm{ \vec{x}_0 }
\end{equation}
for all ${\rb{t,\vec{x}_0}\in\Omega_\delta}$ and ${m\in\setb{1,\ldots,p-1}}$,
where ${y_{01}:=\psiif{1}{\vec{x}_0}}$.
\end{thm}

\begin{proof}
The proof is similar to that of Theorem~\ref{thm001} and omitted in
the interest of space.
\end{proof}

\begin{rem}
Notice in \eqref{eq040} we have $\norm{\vec{x}_0}$ as opposed to
$\norm{\vec{x}_0}^{\alpha+\rb{m-1}\beta}$ which we had in
\eqref{eq020}. This is the result of $\exp{\kappa at}x_{02}$ in
\eqref{eq034} having no corresponding term in \eqref{eq038} with
which to cancel. Note that we cannot merely insert $\exp{\kappa
at}x_{02}$ into the definition of the second component of the
iterates since we need the iterates to be defined independently of
the initial condition $\vec{x}_0$.
\end{rem}

\subsubsection{Second Component of the Transformation}

Let ${y_{01}:=\psiif{1}{\vec{x}_0}}$, where ${\norm{\vec{x}_0}}$ is
sufficiently small. Consider the integral equation \eqref{eq034} for
which the first integral has been ``flipped'' but the second
integral has not. We will manipulate the expression for the second
component of the flow so that we can ``flip some integral'' in some
way. Write
\begin{align}
    \flowCi{2}{t}{\vec{x}_0}
    &= \exp{\kappa a t} \setb{ x_{02} + \integral{0}{t}{ \exp{-\kappa a s} \sqb{ \scif{b}{2}{ \flowC{s}{\vec{x}_0} } - \scif{b}{2}{ \iterateD{p-1}{s}{y_{01}} } } }{s} } \notag\\
    &\quad+ \exp{\kappa a t} \integral{0}{t}{ \exp{-\kappa a s} \scif{b}{2}{ \iterateD{p-1}{s}{y_{01}} } }{s}. \label{eq041}
\end{align}
Before we ``flip the first integral'' in \eqref{eq041}, we need to
check the decay rate of the integrand.

\begin{claim}
Let ${\sigma>0}$ be sufficiently small. There are ${\delta,k>0}$
such that
\begin{equation} \label{eq042}
    \abs{ \exp{-\kappa a t} \sqb{ \scif{b}{2}{ \flowC{t}{\vec{x}_0} } - \scif{b}{2}{ \iterateD{p-1}{t}{y_{01}} } } }
    \leq k \, \exp{\sqb{\alpha+\rb{p-1}\beta-\kappa}\sqb{a+\sigma}t} \norm{\vec{x}_0}^{\beta+1}
\end{equation}
for all ${\rb{t,\vec{x}_0}\in\Omega_\delta}$, where ${y_{01}:=\psiif{1}{\vec{x}_0}}$.
\end{claim}

\begin{proof}
It follows from \eqref{eq004}, \eqref{eq031}, \eqref{eq039}, and
\eqref{eq040} in conjunction with the Mean Value Theorem.
\end{proof}

\begin{claim}
There is a ${\delta>0}$ such that
\begin{align}
    \flowCi{2}{t}{\vec{x}_0}
    &= \sqb{ \exp{\kappa a t} \psiif{2}{\vec{x}_0} + \integral{0}{t}{ \exp{\kappa a\rb{t-s}} \scif{b}{2}{ \iterateD{p-1}{s}{\psiif{1}{\vec{x}_0}} } }{s} } \notag\\
    &\quad- \integral{t}{\infty}{ \exp{\kappa a\rb{t-s}} \sqb{ \scif{b}{2}{ \flowC{s}{\vec{x}_0} } - \scif{b}{2}{ \iterateD{p-1}{s}{\psiif{1}{\vec{x}_0}} } } }{s} \label{eq043}
\end{align}
for all ${\rb{t,\vec{x}_0}\in\Omega_\delta}$, where
\begin{equation} \label{eq044}
    \psiif{2}{\vec{x}_0}
    := x_{02} + \integral{0}{\infty}{ \exp{-\kappa a s} \sqb{ \scif{b}{2}{ \flowC{s}{\vec{x}_0} } - \scif{b}{2}{ \iterateD{p-1}{s}{\psiif{1}{\vec{x}_0}} } } }{s}.
\end{equation}
\end{claim}

\begin{proof}
It follows from \eqref{eq034}, \eqref{eq037}, \eqref{eq041},
\eqref{eq042}, and \eqref{eq044}.
\end{proof}

\begin{claim}
There are ${\delta,k>0}$ such that
\begin{equation} \label{eq045}
    \abs{ \psiif{2}{\vec{x}_0} - x_{02} }
    \leq k \norm{\vec{x}_0}^{\beta+1}
    \mymbox{for all}
    \vec{x}_0 \in \ballO{\delta}.
\end{equation}
\end{claim}

\begin{proof}
It follows from \eqref{eq042} and \eqref{eq044}.
\end{proof}

Notice the definition of $\psiif{2}{\vec{x}_0}$ depends on
$\psiif{1}{\vec{x}_0}$. See \S\ref{sec003.005.002} for how to
approximate
\begin{equation} \label{eq046}
    \vecpsif{\vec{x}_0}
    = \twobyonematrix{ x_{01} + \integral{0}{\infty}{ \exp{-a s} \scif{b}{1}{\flowC{s}{\vec{x}_0}} }{s} }{ x_{02} + \integral{0}{\infty}{ \exp{-\kappa a s} \sqb{ \scif{b}{2}{ \flowC{s}{\vec{x}_0} } - \scif{b}{2}{ \iterateD{p-1}{s}{\psiif{1}{\vec{x}_0}} } } }{s} }.
\end{equation}

\begin{claim}
There are ${\delta,k>0}$ such that
\[
    \norm{ \vecpsif{\vec{x}_0} - \vec{x}_0 }
    \leq k \norm{\vec{x}_0}^{\min{}{\alpha,\beta+1}}
    \mymbox{for all}
    \vec{x}_0 \in \ballO{\delta}.
\]
Furthermore, $\vecpsi$ is a near-identity transformation.
\end{claim}

\begin{proof}
The conclusion follows from \eqref{eq036} and \eqref{eq045} along
with the fact that ${\alpha>1}$ and ${\beta>0}$.
\end{proof}

\subsubsection{Second Group of Iterates}

We are now in a position to define the remainder of the iterates,
$\setb{\iterateD{m}{t}{\vec{y}_0}}_{m=p}^\infty$.  Let
$\norm{\vec{y}_0}$ be sufficiently small. With inspiration from
\eqref{eq034} and \eqref{eq043}, we will define the $p^\text{th}$
iterate by
\begin{subequations} \label{eq047}
\begin{equation} \label{eq047a}
    \iterateD{p}{t}{\vec{y}_0}
    := \twobyonematrix{ \exp{a t} y_{01} - \integral{t}{\infty}{ \exp{a\rb{t-s}} \scif{b}{1}{ \iterateD{p-1}{s}{y_{01}} } }{s} }{ \exp{\kappa a t} y_{02} + \integral{0}{t}{ \exp{\kappa a\rb{t-s}} \scif{b}{2}{ \iterateD{p-1}{s}{y_{01}} } }{s} }
\end{equation}
and define the subsequent iterates recursively by
\begin{equation} \label{eq047b}
    \iterateD{m+1}{t}{\vec{y}_0}
    := \iterateD{p}{t}{\vec{y}_0}- \integral{t}{\infty}{ \exp{\rb{t-s}\mat{A}} \sqb{ \vecf{b}{\iterateD{m}{s}{\vec{y}_0}} - \vecf{b}{\iterateD{p-1}{s}{y_{01}}} } }{s}
    \quad \rb{ m \geq p}.
\end{equation}
\end{subequations}
To connect the iterates with the flow, we choose
${\vec{y}_0:=\vecpsif{\vec{x}_0}}$, where $\vecpsi$ is defined in
\eqref{eq046} and $\norm{\vec{x}_0}$ is sufficiently small. Observe
that we can combine \eqref{eq034}, \eqref{eq043}, and \eqref{eq047a}
to obtain the following.

\begin{claim}
There is a ${\delta>0}$ such that, for all
${\rb{t,\vec{x}_0}\in\Omega_\delta}$,
\[
    \flowC{t}{\vec{x}_0}
    = \iterateD{p}{t}{\vecpsif{\vec{x}_0}} - \integral{t}{\infty}{ \exp{\rb{t-s}\mat{A}} \sqb{ \vecf{b}{\flowC{s}{\vec{x}_0}} - \vecf{b}{\iterateD{p-1}{s}{\psiif{1}{\vec{x}_0}}} } }{s}.
\]
\end{claim}

\subsubsection{Existence, Decay Rate, and Closeness to the Flow of the Second Group}

\begin{prop} \label{prop003}
Consider the transformation defined by \eqref{eq046} and the second
group of iterates defined by \eqref{eq047}. Then, there is a
${\delta>0}$ (independent of $m$) such that if
${\norm{\vec{x}_0}<\delta}$ then
$\iterateD{m}{t}{\vecpsif{\vec{x}_0}}$ exists for each ${m \geq p}$.
Moreover, there is a ${k>0}$ (independent of $m$) such that
\[
    \norm{ \iterateD{m}{t}{\vecpsif{\vec{x}_0}} } \leq k \, \exp{at} \norm{ \vec{x}_0 }
    \mymbox{for all}
    \rb{t,\vec{x}_0} \in \Omega_\delta, ~ m \geq p.
\]
\end{prop}

\begin{proof}
The proof is similar to that of Proposition~\ref{prop001} and omitted in the interest of space.
\end{proof}

\begin{thm} \label{thm003}
Consider the transformation defined by \eqref{eq046} and the second
group of iterates defined by \eqref{eq047}. Let ${\sigma>0}$ be
sufficiently small. Then, there are constants ${\delta>0}$ and
${\seq{k}{m}{p}{\infty} \subset \Rd{}_+}$ such that, for all
${\rb{t,\vec{x}_0}\in\Omega_\delta}$ and ${m \geq p}$,
\[
    \norm{ \flowC{t}{\vec{x}_0} - \iterateD{m}{t}{\vecpsif{\vec{x}_0}} }
    \leq k_m \exp{ \sqb{ \alpha + \rb{ m - 1 } \beta } \sqb{ a + \sigma } t } \norm{ \vec{x}_0 }^{\rb{m+1-p}\beta+1}.
\]
\end{thm}

\begin{proof}
The proof is similar to that of Theorem~\ref{thm001} and omitted in
the interest of space.
\end{proof}

\begin{cor}
Consider the iterates defined by \eqref{eq038} and \eqref{eq047}.
There exists ${\delta>0}$ such that
\dsm{\lim_{m\to\infty}\iterateD{m}{t}{\vecpsif{\vec{x}_0}}=\flowC{t}{\vec{x}_0}}
uniformly in $\Omega_\delta$.
\end{cor}

\begin{proof}
The proof is similar to that of Corollary~\ref{cor001}.
\end{proof}

\subsection{Resonance, Quadratic Nonlinearity, and Relating Components} \label{sec003.003}

The special case which we focus on here is applicable to the
Michaelis-Menten mechanism which we will see later. Assume
${\kappa\in\setb{2,3,\ldots}}$, that is, there is resonance in the
eigenvalues. Define
\[
    \vecf{b}{\vec{x}} := \twobyonematrix{ b_{111} x_1^2 + b_{112} x_1 x_2 + b_{122} x_2^2 }{ b_{211} x_1^2 + b_{212} x_1 x_2 + b_{222} x_2^2 },
\]
where ${\setb{b_{ijk}}_{i,j,k=1}^2\subset\Rd{}}$ are constants with
at least one $b_{ijk}$ being non-zero. Take ${\alpha=2}$,
${\beta=1}$, and ${p=\kappa}$. Hence, consider the initial value
problem
\begin{equation} \label{eq048}
    \dot{\vec{x}} = \mat{A} \vec{x} + \vecf{b}{\vec{x}},
    \quad
    \vec{x}(0) = \vec{x}_0.
\end{equation}
Let ${\vec{y}_0:=\vecpsif{\vec{x}_0}}$, where $\norm{\vec{x}_0}$ is
sufficiently small, satisfy ${y_{01}>0}$. (Since
${x_1(t)=y_{01}\,\exp{at}+\litO{\exp{at}}}$ as ${t\to\infty}$, we
assume ${y_{01}>0}$ since it corresponds to solutions with ${x_1>0}$
approaching the equilibrium point in the slow direction. This is
desirable since it corresponds to, for example, solutions that arise
in chemical kinetics.) Also, let $\vec{x}(t)$ be the solution to
\eqref{eq048}. Our goal is to obtain an asymptotic expression for
$x_2$ in terms of $x_1$ as ${x_1 \to 0^+}$ up to an order for which
the initial condition distinguishes solutions. This can be achieved
without knowing $\vec{y}_0$ in terms of $\vec{x}_0$ explicitly.

\paragraph{$\diamond$ Case 1: ${\kappa=2}$:} We will only need the first two iterates. Using \eqref{eq038a}
and \eqref{eq047a},
\[
    \iterateD{1}{t}{y_{01}} = \twobyonematrix{ y_{01} \, \exp{at} }{0}
    \mymbox{and}
    \iterateD{2}{t}{\vec{y}_0} = \twobyonematrix{ y_{01} \, \exp{at}  + \sqb{ \frac{b_{111}y_{01}^2}{a} } \exp{2at} }{ y_{02} \, \exp{2at}  + \sqb{ b_{211} y_{01}^2 } t \, \exp{2at} }.
\]

\begin{claim}
The components $x_1(t)$ and $x_2(t)$ of the solution to
\eqref{eq048} satisfy
\begin{equation}  \label{eq049}
    x_1(t) = \rb{ y_{01} } \exp{at} + \litO{\exp{at}}
    \mymbox{and}
    x_2(t) = \rb{ b_{211} y_{01}^2 } t \, \exp{2at} + \rb{ y_{02} } \exp{2at} + \litO{\exp{2at}}
    \mymbox{as}
    t \to \infty.
\end{equation}
\end{claim}

\begin{proof}
It follows from Theorem~\ref{thm003} and our expression for the
second iterate.
\end{proof}

\begin{claim}
We can write $\exp{at}$, $\exp{2at}$, and $t$ in terms of $x_1$ when ${x_1 \to 0^+}$ as
\begin{equation} \label{eq050}
    \exp{at} = \rb{ \frac{1}{y_{01}} } x_1 + \litO{x_1},
    \quad
    \exp{2at} = \rb{ \frac{1}{y_{01}^2} } x_1^2 + \litO{x_1^2},
    \mymbox{and}
    t = \rb{ \frac{1}{a} } \lnb{x_1} - \frac{\lnb{y_{01}}}{a} + \litO{1}.
\end{equation}
\end{claim}

\begin{proof}
To obtain the first part, use the first expression of \eqref{eq049}
and a method employed in \S\ref{sec002.004.001}. The second part
follows from the first. To obtain the third part, write the first
part as
${\exp{at}=\sqb{\fracslash{1}{y_{01}}}\sqb{x_1}\sqb{1+\litO{1}}}$ as
${x_1 \to 0^+}$. Taking the natural logarithm of both sides,
\[
    a t = -\lnb{ y_{01} } + \lnb{ x_1 } + \lnb{ 1 + \litO{1} }
    \mymbox{as}
    x_1 \to 0^+.
\]
Observing that ${\lnb{1+u}=\litO{1}}$ as ${u \to 0}$, we are left
with what we were trying to show.
\end{proof}

\begin{prop} \label{prop004}
Consider the initial value problem \eqref{eq048}, where
${\kappa=2}$, $\norm{\vec{x}_0}$ is sufficiently small,
${\vec{y}_0=\vecpsif{\vec{x}_0}}$, and ${\psiif{1}{\vec{x}_0}>0}$.
If $\vec{x}(t)$ is the solution, then we can write $x_2$ in terms of
$x_1$ as
\[
    x_2 = \sqb{ \frac{b_{211}}{a} } x_1^2 \, \lnb{x_1} + \sqb{ \frac{y_{02}}{y_{01}^2} - \frac{ b_{211} \lnb{y_{01}} }{a} } x_1^2 + \litO{x_1^2}
    \mymbox{as}
    x_1 \to 0^+.
\]
\end{prop}

\begin{proof}
Substitute the second and third of \eqref{eq050} into the second of
\eqref{eq049} then simplify.
\end{proof}

\paragraph{$\diamond$ Case 2: ${\kappa\in\setb{3,4,\ldots}}$:} Using \eqref{eq038},
\[
    \iterateD{1}{t}{y_{01}} = \twobyonematrix{ y_{01} \, \exp{at} }{ 0 }
    \mymbox{and}
    \iterateD{2}{t}{y_{01}} = \twobyonematrix{ y_{01} \, \exp{at} + \sqb{ \frac{b_{111}y_{01}^2}{a} } \exp{2at} }{ \sqb{ \frac{b_{211}y_{01}^2}{\rb{2-\kappa}a} } \exp{2at} - \sqb{ \frac{b_{211}y_{01}^2}{\rb{2-\kappa}a} } \exp{\kappa at} }.
\]
By Theorem~\ref{thm002},
\[
    x_1(t) = y_{01} \, \exp{at} + \rb{ \frac{b_{111}y_{01}^2}{a} } \exp{2at} + \bigO{\exp{3at}}
    \mymbox{and}
    x_2(t) = \sqb{ \frac{b_{211}y_{01}^2}{\rb{2-\kappa}a} } \exp{2at} + \bigO{\exp{3at}}
\]
as ${t\to\infty}$. By inspection, ${x_2=\bigO{x_1^2}}$ as ${x_1 \to 0^+}$. To take this further, we need the following.

\begin{claim} \label{claim003}
There exist constants $\seq{\xi}{i}{1}{\kappa}$ and $\seq{\rho}{i}{2}{\kappa-1}$ (which do not depend on the initial
condition) and a constant $\varrho$ (which may depend on the initial
condition) such that, as ${t\to\infty}$,
\begin{subequations} \label{eq051}
\begin{align}
    x_1(t) &= \sum_{i=1}^\kappa \xi_i \rb{ \exp{at} y_{01} }^i + \bigO{\exp{\rb{\kappa+1}at}},
    \mymbox{with}
    \xi_1 = 1, \label{eq051a} \\
    \mymbox{and}
    x_2(t) &= \sum_{i=2}^{\kappa-1} \rho_i \rb{ \exp{at} y_{01} }^i + \varrho \rb{ \exp{\kappa at} } + \bigO{\exp{\rb{\kappa+1}at}},
    \mymbox{with}
    \rho_2 = \frac{b_{211}}{\rb{2-\kappa}a}. \label{eq051b}
\end{align}
\end{subequations}
\end{claim}

\begin{proof}
It follows from Theorems~\ref{thm002} and \ref{thm003}, the
definition \eqref{eq038} of the first ${\kappa-1}$ iterates, and the
definition \eqref{eq047a} of the $\kappa^\text{th}$ iterate.
\end{proof}

\begin{prop} \label{prop005}
Consider the initial value problem \eqref{eq048}, where
${\kappa\in\setb{3,4,\ldots}}$, $\norm{\vec{x}_0}$ is sufficiently
small, ${\vec{y}_0=\vecpsif{\vec{x}_0}}$, and
${\psiif{1}{\vec{x}_0}>0}$. If $\vec{x}(t)$ is the solution, then
for some constants $\seq{c}{i}{2}{\kappa-1}$ (which depend only on
the differential equation) and $C$ (which depends on the
differential equation and the initial condition) we can write $x_2$
in terms of $x_1$ as
\begin{equation} \label{eq052}
    x_2 = \sum_{i=2}^{\kappa-1} c_i \, x_1^i + C \, x_1^\kappa + \bigO{x_1^{\kappa+1}}
    \mymbox{as}
    x_1 \to 0^+,
    \mymbox{with}
    c_2 = \frac{b_{211}}{\rb{2-\kappa}a}.
\end{equation}
\end{prop}

\begin{proof}
Note that ${x_1=\bigO{\exp{at}}}$ as ${t\to\infty}$ and
${\exp{at}=\bigO{x_1}}$ as ${x_1 \to 0^+}$, which follow from
\eqref{eq051a}. Now, by inverting \eqref{eq051a} in a process we
have done many times before, it follows that there are constants
$\seq{\nu}{i}{1}{\kappa}$ which do not depend on the initial
condition such that
\begin{equation} \label{eq053}
    \exp{at} y_{01} = \sum_{i=1}^\kappa \nu_i \, x_1^i + \bigO{x_1^{\kappa+1}}
    \mymbox{as}
    x_1 \to 0^+,
    \mymbox{with}
    \nu_1 = 1.
\end{equation}
Substituting \eqref{eq053} in \eqref{eq051b},  we are left with
\eqref{eq052} for some constants $\seq{c}{i}{2}{\kappa-1}$ (which do
not depend on the initial condition) and $C$ (which may depend on
the initial condition).
\end{proof}

\subsection{Generalizing to a Diagonalizable Matrix}

Suppose $\mat{A}$ and $\vec{b}$ are as before except $\mat{A}$ is
only assumed to be diagonalizable with (fast and slow, respectively)
eigenvalues ${\lambda_1=a}$ and ${\lambda_2=\kappa a}$. Consider the
initial value problem
\begin{equation} \label{eq054}
    \dot{\vec{x}} = \mat{A} \vec{x} + \vecf{b}{\vec{x}},
    \quad
    \vec{x}(0) = \vec{x}_0.
\end{equation}
Now, there exists invertible matrix $\mat{P}$ such that
${\matLambda=\mat{P}^{-1}\mat{A}\mat{P}}$, where
${\matLambda:=\diag{a,\kappa a}}$. Define
${\vecf{r}{\vec{u}}:=\mat{P}^{-1}\vecf{b}{\vec{P}\vec{u}}}$ and
${\vec{u}_0:=\mat{P}^{-1}\vec{x}_0}$. The following is easy to
verify.

\begin{claim}
If $\vec{x}(t)$ is a solution of \eqref{eq054}, then
${\vec{u}(t):=\mat{P}^{-1}\vec{x}(t)}$ is a solution of
\begin{equation} \label{eq055}
    \dot{\vec{u}} = \matLambda \vec{u} + \vecf{r}{\vec{u}},
    \quad
    \vec{u}(0) = \vec{u}_0.
\end{equation}
Moreover, there are ${\delta,k_1,k_2>0}$ such that
${\norm{\vecf{r}{\vec{u}}} \leq k_1\norm{\vec{u}}^\alpha}$ and
${\norm{\mat{D}\vecf{r}{\vec{u}}} \leq k_2\norm{\vec{u}}^\beta}$ for
all ${\vec{u}\in\ballO{\delta}}$.
\end{claim}

\subsubsection{Associated Iterates}

The iterates $\setb{\iterateD{m}{t}{\vecpsif{\vec{u}_0}}}_{m=1}^\infty$, defined with respect to $\matLambda$ and $\vec{r}$, allow us to construct \emph{associated iterates}, $\setb{\mat{P}\iterateD{m}{t}{\vecpsif{\mat{P}^{-1}\vec{x}_0}}}_{m=1}^\infty$, for the original system.

\begin{thm}
Let $\vec{x}(t)$ be the solution of the initial value problem
\eqref{eq054}. Consider the iterates given by \eqref{eq038} and
\eqref{eq047} and the transformation given by \eqref{eq046}, defined
with respect to \eqref{eq055}. For any sufficiently small
${\sigma>0}$, there exist ${\delta>0}$ and ${\seq{k}{m}{1}{\infty}
\subset \Rd{}_+}$ such that, for all
${\rb{t,\vec{x}_0}\in\Omega_\delta}$ and ${m\in\Nd{}}$,
\[
    \norm{ \vec{x}(t) - \mat{P} \iterateD{m}{t}{\vecpsif{\mat{P}^{-1}\vec{x}_0}} }
    \leq k_m \exp{ \sqb{ \alpha + \rb{ m - 1 } \beta } \sqb{ a + \sigma } t } \norm{ \vec{x}_0 }^{\max{}{0,m+1-p}\beta+1}.
\]
\end{thm}

\begin{proof}
Let ${\delta>0}$ be sufficiently small so that if
${\norm{\vec{x}_0}<\delta}$ then the appropriate estimates apply.
Hence, assume ${\vec{x}_0\in\ballO{\delta}}$. Let
${\vec{u}(t):=\mat{P}^{-1}\vec{x}(t)}$, which is the solution to
\eqref{eq055}. Observe that
\[
    \norm{\vec{u}_0} \leq \norm{\mat{P}^{-1}} \norm{\vec{x}_0}
    \mymbox{and}
    \norm{ \vec{x}(t) - \mat{P} \iterateD{m}{t}{\vecpsif{\vec{u}_0}} }
    \leq \norm{\mat{P}} \norm{ \vec{u}(t) - \iterateD{m}{t}{\vecpsif{\vec{u}_0}} }.
\]
Applying Theorems~\ref{thm002} and \ref{thm003} give the conclusion.
\end{proof}

\subsubsection{Resonance, Quadratic Nonlinearity, and Relating
Components} \label{sec003.004.002}

Consider the initial value problem \eqref{eq054} with the particular
nonlinear part
\begin{equation} \label{eq056}
    \vecf{b}{\vec{x}} := \twobyonematrix{ b_{111} x_1^2 + b_{112} x_1 x_2 + b_{122} x_2^2 }{ b_{211} x_1^2 + b_{212} x_1 x_2 + b_{222} x_2^2 },
\end{equation}
where ${\setb{b_{ijk}}_{i,j,k=1}^2\subset\Rd{}}$ are constants with
at least one $b_{ijk}$ being non-zero. The nonlinear part for the
diagonalized initial value problem \eqref{eq055} is of the form,
where $\setb{r_{ijk}}_{i,j,k=1}^2$ are constants,
\[
    \vecf{r}{\vec{u}} = \twobyonematrix{ r_{111} u_1^2 + r_{112} u_1 u_2 + r_{122} u_2^2 }{ r_{211} u_1^2 + r_{212} u_1 u_2 + r_{222} u_2^2 }.
\]

We want to be able to write $x_2$ as a function of $x_1$ when
$\vec{x}(t)$ approaches the origin in the slow direction and there
is resonance in the eigenvalues, that is, when
${\kappa\in\setb{2,3,\ldots}}$. (The no-resonance case
${\kappa\not\in\setb{2,3,\ldots}}$ can be handled using iterates but
also by using Poincar\'{e}'s Theorem. See, for example, Theorem~19
of the authors' \cite{CalderSiegel}.) Observe that the material of
\S\ref{sec003.003} applies to \eqref{eq055}.

Since $\vec{x}(t)$ approaches the origin in the slow direction from
the right, we can assume ${p_{11}>0}$ and ${v_{01}>0}$, where
${\vec{v}_0:=\vecpsif{\vec{u}_0}}$ with $\vecpsi$ being defined with
respect to the diagonalized system \eqref{eq055}. To see why we can
take ${p_{11}>0}$, observe that we can take the first column of
$\mat{P}$ to be the slow eigenvector and we can take the slow
eigenvector to have positive first component. Note the slope of the
slow eigenvector is $\fracslash{p_{21}}{p_{11}}$, so we expect
${x_2=\rb{\fracslash{p_{21}}{p_{11}}}x_1+\litO{x_1}}$ as ${x_1 \to
0^+}$. To see why we can take ${v_{01}>0}$, observe
\[
    u_1(t) = v_{01} \, \exp{at} + \litO{\exp{at}},
    \quad
    u_2(t) = \litO{\exp{at}},
    \mymbox{and}
    x_1(t) = p_{11} v_{01} \, \exp{at} + \litO{\exp{at}}
\]
as ${t\to\infty}$, the last of which uses the relationship ${x_1=p_{11}u_1+p_{12}u_2}$.

\begin{thm} \label{thm004}
Consider the initial value problems \eqref{eq054}, with
$\vecf{b}{\vec{x}}$ given by \eqref{eq056}, and \eqref{eq055}.
Suppose ${\kappa\in\setb{2,3,\ldots}}$, ${p_{11}>0}$,
$\norm{\vec{x}_0}$ is sufficiently small, and ${v_{01}>0}$ where
${\vec{v}_0:=\vecpsif{\vec{u}_0}}$ with $\vecpsi$ as in
\eqref{eq046} defined with respect to \eqref{eq055}. Let
$\vec{x}(t)$ be the solution to \eqref{eq054}.
\begin{enumerate}[(a)]
    \item
        If ${\kappa=2}$, then we can write $x_2$ in terms of $x_1$
        as
        \begin{equation} \label{eq057}
            x_2
            = \sqb{ \frac{p_{21}}{p_{11}} } x_1 + \sqb{ \frac{r_{211}\det{\mat{P}}}{p_{11}^3a} } x_1^2 \, \lnb{x_1}
            + \sqb{ \frac{\det{\mat{P}}}{p_{11}^3} } \sqb{ \frac{v_{02}}{v_{01}^2} - \frac{ r_{211} \lnb{p_{11}v_{01}} }{a} } x_1^2 + \litO{x_1^2}
            \mymbox{as}
            x_1 \to 0^+.
        \end{equation}
    \item
        If ${\kappa\in\setb{3,4,\ldots}}$, then we can write $x_2$ in terms of $x_1$
        as
        \begin{equation} \label{eq058}
            x_2 = \sqb{ \frac{p_{21}}{p_{11}} } x_1 + \sum_{i=2}^{\kappa-1} c_i \, x_1^i + C \, x_1^\kappa + \bigO{x_1^{\kappa+1}}
            \mymbox{as}
            x_1 \to 0^+
        \end{equation}
        for some constants $\seq{c}{i}{2}{\kappa-1}$ (which depend only on the
        differential equation) and $C$ (which depends on the differential
        equation and the initial condition).
\end{enumerate}
\end{thm}

\begin{proof}
As mentioned above, $\vec{x}(t)$ approaches the origin in the slow direction and is
strictly positive for sufficiently large $t$.
\begin{enumerate}[(a)]
    \item
        Assume we can write $u_2$ as a function of $u_1$, say ${u_2=f(u_1)}$. Then, ${x_1=p_{11}u_1+p_{12}f(u_1)}$ since ${\vec{x}=\mat{P}\vec{u}}$. Assume further that we can solve this relation for $u_1$ in terms of $x_1$, say ${u_1=g(x_1)}$. Thus, since ${\vec{x}=\mat{P}\vec{u}}$ we have ${x_2=h(x_1)}$, where ${h(x_1):=p_{21}g(x_1)+p_{22}f(g(x_1))}$. With this in mind, the relation \eqref{eq057}, after routine simplification, follows from Proposition~\ref{prop004} and techniques already employed in this paper. Note that ${\det{\mat{P}} \equiv p_{11}p_{22}-p_{12}p_{21} \ne 0}$.
    \item
        The relation \eqref{eq058} follows from Claim~\ref{claim003}, the fact ${\vec{x}=\mat{P}\vec{u}}$, and the method used in the proof of Proposition~\ref{prop005}.
\end{enumerate}
\end{proof}

\begin{rem}
Suppose that, in Theorem~\ref{thm004}, the components of
$\vecf{b}{\vec{x}}$ are polynomial (instead of purely quadratic)
with lowest order ${\alpha\in\setb{2,3,\ldots}}$. If
${\kappa=\alpha}$, then the expression for $x_2$ in terms of $x_1$
as ${x_1 \to 0^+}$ will contain $\lnb{x_1}$. If ${\kappa>\alpha}$,
then the expression for $x_2$ will be a Taylor series in $x_1$ with
${Cx_1^\kappa}$ being the lowest-order term which depends on the
initial condition.
\end{rem}

\subsection{Generalizing to Higher Dimensions} \label{sec003.005}

\subsubsection{Introduction}

The techniques of this section can be generalized to the
$n$-dimensional case where the eigenvalues of $\mat{A}$ are
widely-spaced relative to $\vecf{b}{\vec{x}}$. The derivations and
proofs will be omitted since they are tedious and reminiscent.
Consider the general initial value problem \eqref{eq001} with ${n
\geq 2}$ and ${\kappa\geq\alpha}$. Suppose that ${\mat{J}\in\Rd{n
\times n}}$ is the real Jordan canonical form of $\mat{A}$ which
satisfies the same ordering \eqref{eq003} for the eigenvalues
$\seq{\lambda}{i}{1}{n}$. Then, there exists invertible matrix
${\mat{P}\in\Rd{n \times n}}$ such that
${\mat{J}=\mat{P}^{-1}\mat{A}\mat{P}}$.

We need to partition the Jordan blocks of $\mat{J}$. Suppose the
eigenvalues of $\mat{J}$ have ${\ell\in\setb{2,\ldots,n}}$ distinct
real parts $\seq{\hmu}{j}{1}{\ell}$, each with multiplicity $n_j$ so
that ${\sum_{j=1}^\ell n_j=n}$, which are ordered according to
${\hmu_\ell\leq\cdots\leq\hmu_1<0}$. Then, we can write
${\mat{J}=\bigoplus_{j=1}^\ell \mat{J}_j}$ where each
${\mat{J}_j\in\Rd{n_j \times n_j}}$ is the direct sum of all real
Jordan blocks having associated eigenvalues with real part $\hmu_j$.
Note that ${\hmu_1=\mu_1}$ and ${\hmu_\ell=\mu_n}$. We will write
all vectors ${\vec{x}\in\Rn}$ in component form as
${\vec{x}=\rb{\vec{x}_1,\ldots,\vec{x}_\ell}^T}$ with
${\vec{x}_j\in\Rd{n_j}}$ for each ${j\in\setb{1,\ldots,\ell}}$.

Hence, consider the initial value problem
\begin{equation} \label{eq059}
    \dot{\vec{u}} = \mat{J} \vec{u} + \vecf{r}{\vec{u}},
    \quad
    \vec{u}(0) = \vec{u}_0,
\end{equation}
where ${\vec{u}:=\mat{P}^{-1}\vec{x}}$,
${\vec{u}_0:=\mat{P}^{-1}\vec{x}_0}$, and
${\vecf{r}{\vec{u}}:=\mat{P}^{-1}\vecf{b}{\mat{P}\vec{u}}}$. Denote
the flow of \eqref{eq059} by $\flowE{t}{\vec{u}_0}$ and observe each
component satisfies the integral equation
\[
    \flowEj{j}{t}{\vec{u}_0} = \exp{t\mat{J}_j} \vec{u}_{0j} + \integral{0}{t}{ \exp{\rb{t-s}\mat{J}_j} \vecif{r}{j}{\flowE{s}{\vec{u}_0}} }{s}
    \quad
    \rb{ j = 1, \ldots, \ell }.
\]

Define
\[
    \kappa_j := \frac{\hmu_j}{\hmu_1}
    \quad \rb{ j = 1, \ldots, \ell }
    \mymbox{and}
    j_0 := \min{}{ \, j \in \setb{1,\ldots,\ell} \st \kappa_j \geq \alpha \, },
\]
and note ${1=\kappa_1<\kappa_2<\cdots<\kappa_\ell=\kappa}$ and
${j_0\in\setb{2,\ldots,\ell}}$. Observe ${1 \leq j < j_0}$
corresponds to blocks which are closely-spaced relative to the
nonlinear part. Define also
\[
    p_j :=
    \begin{cases}
        1, \mymbox{if} 1 \leq j < j_0 \\
        \ds{ \floor{ \frac{\kappa_j-\alpha}{\beta} } + 2 }, \mymbox{if} j_0 \leq j \leq \ell
    \end{cases}
    \quad \rb{ j = 1, \ldots, \ell }.
\]
Observe ${2 \leq p_{j_0} \leq \cdots \leq p_\ell}$ and
\[
    \alpha + \rb{p_j-2} \beta \leq \kappa_j < \alpha + \rb{p_j-1} \beta
    \quad
    \rb{ j = j_0, \ldots, \ell }.
\]
Note that it can be easily shown, for any ${\sigma>0}$, there exist
${\delta>0}$ and ${\seq{k}{j}{j_0}{\ell}\subset\Rd{}_+}$ such that
${\norm{\flowEj{j}{t}{\vec{u}_0}} \leq
k_j\,\exp{\rb{\alpha\mu_1+\sigma}t}\norm{\vec{u}_0}}$ for all
${\rb{t,\vec{u}_0}\in\Omega_\delta}$ and
${j\in\setb{j_0,\ldots,\ell}}$. This estimate is required to
establish the base case of Theorem~\ref{thm005} below.

\subsubsection{Iterates and the Transformation} \label{sec003.005.002}

Assume that $\norm{\vec{v}_0}$ is sufficiently small so that we can
define the iterates
$\setb{\iterateD{m}{t}{\vec{v}_0}}_{m=1}^\infty$. The definition,
given in Table~\ref{tab001}, must be split into many cases.

\begin{table}[t]
\begin{center}\setlength{\extrarowheight}{3pt}
\begin{small}
\begin{tabular}{c|c}
    \textbf{Component $j$} & \multirow{2}{*}{ \textbf{Definition of } \dsm{ \iterateDj{j}{m}{t}{\vec{v}_0} } } \\ \textbf{Index $m$} \\\hline\hline
    ${1 \leq j < j_0}$ & \multirow{2}{*}{ \dsm{ \exp{t\mat{J}_j} \vec{v}_{0j} } } \\ ${m=1}$ \\\hline
    ${1 \leq j < j_0}$ & \multirow{2}{*}{ \dsm{ \exp{t\mat{J}_j} \vec{v}_{0j} - \integral{t}{\infty}{ \exp{\rb{t-s}\mat{J}_j} \vecif{r}{j}{\iterateD{m-1}{s}{\vec{v}_0}} }{s} } } \\ ${m \geq 2}$ \\\hline
    ${j_0 \leq j \leq \ell}$ & \multirow{2}{*}{ \dsm{ \vec{0} } } \\ ${m=1}$ \\\hline
    ${j_0 \leq j \leq \ell}$ & \multirow{2}{*}{ \dsm{ \integral{0}{t}{ \exp{\rb{t-s}\mat{J}_j} \vecif{r}{j}{\iterateD{m-1}{s}{\vec{v}_0}} }{s} } } \\ ${2 \leq m < p_j}$ \\\hline
    ${j_0 \leq j \leq \ell}$ & \multirow{2}{*}{ \dsm{ \exp{t\mat{J}_j} \vec{v}_{0j} + \integral{0}{t}{ \exp{\rb{t-s}\mat{J}_j} \vecif{r}{j}{\iterateD{p_j-1}{s}{\vec{v}_0}} }{s} } } \\ ${m=p_j}$ \\\hline
    ${j_0 \leq j \leq \ell}$ & \multirow{2}{*}{ \dsm{ \iterateDj{j}{p_j}{t}{\vec{v}_0} - \integral{t}{\infty}{ \exp{\rb{t-s}\mat{J}_j} \sqb{ \vecif{r}{j}{\iterateD{m-1}{s}{\vec{v}_0}} - \vecif{r}{j}{\iterateD{p_j-1}{s}{\vec{v}_0}} } }{s} } } \\ ${m>p_j}$
\end{tabular}
\end{small}
\caption{Definition of the iterates
$\setb{\iterateD{m}{t}{\vec{v}_0}}_{m=1}^\infty$ for the initial
value problem \eqref{eq059}.} \label{tab001}
\end{center}
\end{table}

\begin{rem}
For a particular ${j\in\setb{j_0,\ldots,\ell}}$, $\vec{v}_{0j}$ only appears in the iterates $\setb{\iterateD{m}{t}{\vec{v}_0}}_{m=p_j}^\infty$.
\end{rem}

To connect the iterates with the flow, where $\norm{\vec{u}_0}$ is
sufficiently small, take ${\vec{v}_0:=\vecpsif{\vec{u}_0}}$, where
\begin{equation}  \label{eq060}
    \vecpsiif{j}{\vec{u}_0}
    :=
    \begin{cases}
        \vec{u}_{0j} + \integral{0}{\infty}{ \exp{-s\mat{J}_j} \vecif{r}{j}{\flowE{s}{\vec{u}_0}} }{s}, \mymbox{if} 1 \leq j < j_0 \\
        \vec{u}_{0j} + \integral{0}{\infty}{ \exp{-s\mat{J}_j} \sqb{ \vecif{r}{j}{\flowE{s}{\vec{u}_0}} - \vecif{r}{j}{\iterateD{p_j-1}{s}{\vecpsif{\vec{u}_0}}} } }{s}, \mymbox{if} j_0 \leq j \leq \ell
    \end{cases}.
\end{equation}
It can be easily shown that, for some constants ${\delta,k>0}$, the
transformation $\vecpsi$ satisfies
\[
    \norm{ \vecpsif{\vec{u}_0} - \vec{u}_0 } \leq k \norm{ \vec{u}_0 }^{\min{}{\alpha,\beta+1}}
    \mymbox{for all}
    \vec{u}_0 \in \ballO{\delta}.
\]

\begin{rem}
The term $\iterateD{p_j-1}{s}{\vecpsif{\vec{u}_0}}$ inside the
integral in \eqref{eq060} for the case ${j_0 \leq j \leq \ell}$ does
not depend on the components
$\setb{\vecpsiif{i}{\vec{u}_0}}_{i=j}^\ell$. Thus, the definition is
explicit.
\end{rem}

The iterates given in Table~\ref{tab001}, which are expressed in
terms of the unknown parameter $\vec{v}_0$, are useful for obtaining
the form of the asymptotic expansion for the flow
$\flowE{t}{\vec{u}_0}$ as ${t\to\infty}$ when
${\vec{v}_0=\vecpsif{\vec{u}_0}}$. Fortunately, just like with
closely-spaced eigenvalues, we can approximate
$\vecpsif{\vec{u}_0}$. The definition of the approximations
$\setb{\vecpsimf{m}{\vec{u}_0}}_{m=1}^\infty$ is given in
Table~\ref{tab002}.

\begin{table}[t]
\begin{center}\setlength{\extrarowheight}{3pt}
\begin{small}
\begin{tabular}{c|c}
    \textbf{Component $j$} & \multirow{2}{*}{ \textbf{Definition of } \dsm{ \vecpsiimf{j}{m}{\vec{u}_0} } } \\ \textbf{Index $m$} \\\hline\hline
    ${1 \leq j \leq \ell}$ & \multirow{2}{*}{ \dsm{ \vec{u}_{0j} } } \\ ${m=1}$ \\\hline
    ${1 \leq j < j_0}$ & \multirow{2}{*}{ \dsm{ \vec{u}_{0j} + \integral{0}{\infty}{ \exp{-s\mat{J}_j} \vecif{r}{j}{\iterateD{m+p_\ell-2}{s}{\vecpsimf{m-1}{\vec{u}_0}}} }{s} } } \\ ${m \geq 2}$ \\\hline
    ${j_0 \leq j \leq \ell}$ & \multirow{2}{*}{ \dsm{ \vec{u}_{0j} + \integral{0}{\infty}{ \exp{-s\mat{J}_j} \sqb{ \vecif{r}{j}{\iterateD{m+p_\ell-2}{s}{\vecpsimf{m-1}{\vec{u}_0}}} - \vecif{r}{j}{\iterateD{p_j-1}{s}{\vecpsimf{m-1}{\vec{u}_0}}} } }{s} } } \\ ${m \geq 2}$
\end{tabular}
\end{small}
\caption{Definition of the approximations
$\setb{\vecpsimf{m}{\vec{u}_0}}_{m=1}^\infty$ to the transformation
$\vecpsif{\vec{u}_0}$ for the initial value problem \eqref{eq059}.}
\label{tab002}
\end{center}
\end{table}

\subsubsection{Results}

\begin{prop}
Consider the iterates defined in Table~\ref{tab001} with respect to
the initial value problem \eqref{eq059} and let ${\sigma>0}$ be
sufficiently small. Then, there is a ${\delta>0}$ (independent of
$m$) such that if ${\norm{\vec{u}_0}<\delta}$ then
$\iterateD{m}{t}{\vecpsif{\vec{u}_0}}$ exists for each
${m\in\Nd{}}$, where $\vecpsif{\vec{u}_0}$ is as in \eqref{eq060}.
Moreover, there is a ${k>0}$ (independent of $m$) such that
\begin{equation} \label{eq061}
    \norm{ \iterateD{m}{t}{\vecpsif{\vec{u}_0}} } \leq k \, \exp{ \rb{\mu_1+\sigma} t } \norm{\vec{u}_0}
    \mymbox{for all}
    \rb{t,\vec{u}_0} \in \Omega_\delta, ~ m \in \Nd{}.
\end{equation}
\end{prop}

\begin{thm} \label{thm005}
Consider the iterates defined in Table~\ref{tab001} with respect to
the initial value problem \eqref{eq059} and let ${\sigma>0}$ be
sufficiently small. Then, there are constants ${\delta>0}$ and
${\seq{k}{m}{1}{\infty}\subset\Rd{}_+}$ such that
\begin{equation} \label{eq062}
    \norm{ \flowC{t}{\vec{x}_0} - \mat{P} \iterateD{m}{t}{\vecpsif{\mat{P}^{-1}\vec{x}_0}} }
    \leq k_m \, \exp{ \sqb{ \alpha + \rb{m-1} \beta } \sqb{ \mu_1 + \sigma } t } \norm{\vec{x}_0}^{\max{}{0,m+1-p_\ell}\beta+1}
\end{equation}
for all ${\rb{t,\vec{x}_0}\in\Omega_\delta}$ and ${m\in\Nd{}}$,
where $\flowC{t}{\vec{x}_0}$ is the solution of the initial value
problem \eqref{eq001} and $\vecpsif{\vec{u}_0}$ is as in
\eqref{eq060}.
\end{thm}

\begin{rem}
If $\mat{A}$ is diagonalizable and ${\kappa_{j_0}>\alpha}$ then we
can take ${\sigma=0}$ in \eqref{eq061} and \eqref{eq062}.
\end{rem}

\begin{prop}
Consider the transformation defined by \eqref{eq060} and the
approximations defined in Table~\ref{tab002}. There are constants
${\delta>0}$ and ${\seq{k}{m}{1}{\infty}\subset\Rd{}_+}$ such that
\[
    \norm{ \vecpsif{\vec{u}_0} - \vecpsimf{m}{\vec{u}_0} }
    \leq k_m \norm{ \vec{u}_0 }^{\min{}{\alpha,\beta+1}+\rb{m-1}\beta}
    \mymbox{for all}
    \vec{u}_0 \in \ballO{\delta}.
\]
\end{prop}

\section{Application to the Michaelis-Menten Mechanism}
\label{sec004}

The Michaelis-Menten mechanism of an enzyme reaction will be our main application. Particularly, we will
specify the asymptotics of the scalar reduction for small
substrate concentration.

\subsection{Introduction}

The Michaelis-Menten mechanism, named after two researchers
who worked on the model, is the simplest chemical network modeling
the formation of a product from a substrate with the aid of an
enzyme. See, for example,
\cite{Cornish-Bowden,Henri,KeenerSneyd,LaidlerBunting,MichaelisMenten}.
There are different and more complicated networks involving, for
example, multiple enzymes, inhibition, and cooperativity. However,
even though the Michaelis-Menten mechanism is relatively simple,
there are many results on the mechanism and the standard
techniques of analyzing the reaction are easily adapted to the more
complicated networks.

In the scheme of Henri, later explored by Michaelis and
Menten, an enzyme $E$ reacts with the substrate $S$ and reversibly forms an
intermediate complex $C$, which decays into the product $P$ and original
enzyme. With reaction-rate constants $k_{-1}$, $k_1$, and $k_2$, this can be written symbolically as
\[
    S + E \revreaction{k_1}{k_{-\!1}} C \reaction{k_2} P + E.
\]
Technically, the reaction from $C$ to ${P+E}$ is reversible, but the
rate constant $k_{-\!2}$ is usually negligible and the reverse
reaction is omitted. Henri's original model allowed for the reverse
reaction. Moreover, in a cell or laboratory the product is harvested preventing the
reverse reaction.

Let lower-case letters stand for concentrations and $\tau$ for time.
The Law of Mass Action gives
\begin{equation} \label{eq063}
    \deriv{s}{\tau} = k_{-\!1} c - k_1 s e,
    \quad
    \deriv{e}{\tau} = \rb{ k_{-\!1} + k_2 } c - k_1 s e,
    \quad
    \deriv{c}{\tau} = k_1 s e - \rb{ k_{-\!1} + k_2 } c,
    \quad
    \deriv{p}{\tau} = k_2 c.
\end{equation}
Here, the initial conditions are arbitrary. Traditionally, ${s(0)=s_0}$, ${e(0)=e_0}$, and
${c(0)=0=p(0)}$.

Observe that, using \eqref{eq063}, the concentrations $s(\tau)$,
$c(\tau)$, and $e(\tau)$ do not depend on the value of $p(\tau)$.
Moreover, it follows from \eqref{eq063} that $c(\tau)$ and $e(\tau)$
satisfy the conservation law ${c(\tau)+e(\tau)=e_0}$. We can hence
consider only the system of two ordinary differential equations
\begin{equation} \label{eq064}
    \deriv{s}{\tau} = k_{-\!1} c - k_1 s \rb{ e_0 - c },
    \quad
    \deriv{c}{\tau} = k_1 s \rb{ e_0 - c } - \rb{ k_{-\!1} + k_2 } c.
\end{equation}
This is frequently referred to as the planar reduction of the
Michaelis-Menten mechanism. The traditional initial conditions are
${s(0)=s_0}$ and ${c(0)=0}$. However, we will again allow the
initial conditions to be arbitrary. The planar reduction
\eqref{eq064} can be reduced as well to the scalar reduction
\begin{equation} \label{eq065}
    \deriv{c}{s} = \frac{ k_1 s \rb{ e_0 - c } - \rb{ k_{-\!1} + k_2 } c }{ k_{-\!1} c - k_1 s \rb{ e_0 - c } }.
\end{equation}

We will make one last transformation. Define
\[
    t := k_1 e_0 \tau,
    \quad
    x := \frac{ k_1 s }{ k_{-\!1} + k_{2} },
    \quad
    y := \frac{c}{e_0},
    \quad
    \vep := \frac{ k_1 e_0 }{ k_{-\!1} + k_2 },
    \mymbox{and}
    \eta := \frac{ k_2 }{ k_{-\!1} + k_2 },
\]
which are all dimensionless, and note ${\vep>0}$ and ${0<\eta<1}$.
(We make no assumption on the size of $\vep$, which is traditionally
assumed to be small.) Thus, $t$, $x$, and $y$ are, respectively, a
scaled time, substrate concentration, and complex concentration.
Observe ${x=\fracslash{s}{K_m}}$ and ${\vep=\fracslash{e_0}{K_m}}$,
where ${K_m:=\fracslash{\rb{k_{-\!1}+k_2}}{k_1}}$ is the Michaelis
constant (the half-saturation concentration of the substrate). It is
easy to verify that the planar differential equation \eqref{eq064}
assumes the dimensionless form
\begin{equation} \label{eq066}
    \dot{x} = -x + \rb{ 1 - \eta + x } y,
    \quad
    \dot{y} = \vep^{-1} \sqb{ x - \rb{ 1 + x } y },
\end{equation}
where $\dot{}=\derivslash{}{t}$. Moreover, we can also write the
scalar differential equation \eqref{eq065} in the form
\begin{equation} \label{eq067}
    y' = \frac{ x - \rb{ 1 + x } y }{ \vep \sqb{ -x + \rb{ 1 - \eta + x } y } },
    \mymbox{where}
    ' := \deriv{}{x}.
\end{equation}

\subsection{Behaviour at the Origin}

The planar reduction \eqref{eq066}, along with initial conditions
${x(0)=x_0}$ and ${y(0)=y_0}$, can be written
\begin{equation} \label{eq068}
    \dot{\vec{x}} = \mat{A} \vec{x} + \vecf{b}{\vec{x}},
    \quad
    \vec{x}(0) = \vec{x}_0,
    \mymbox{where}
    \mat{A} := \twobytwomatrix{-1}{1-\eta}{\vep^{-1}}{-\vep^{-1}}
    \mymbox{and}
    \vecf{b}{\vec{x}} := x y \twobyonematrix{1}{-\vep^{-1}}.
\end{equation}
The matrix $\mat{A}$ has two distinct, real, negative eigenvalues
given by
\[
    \lambdapm := \frac{ -\rb{ \vep + 1 } \pm \sqrt{ \rb{ \vep + 1 }^2 - 4 \vep \eta } }{ 2 \vep }.
\]
Here, $\lambdap$ is the slow eigenvalue and $\lambdam$ is the fast
eigenvalue. It can be shown that the eigenvalues satisfy
${\lambdam<-1<-\eta<\lambdap<0}$. Moreover, it can be shown that the
nonlinear part $\vecf{b}{\vec{x}}$ satisfies
\[
    \norm{\vecf{b}{\vec{x}}} \leq \rb{ \shalf \sqrt{1+\vep^{-2}} } \norm{\vec{x}}^2
    \mymbox{and}
    \norm{\mat{D}\vecf{b}{\vec{x}}} \leq \rb{\sqrt{1+\vep^{-2}}} \norm{\vec{x}}.
\]

Define
\[
    \sigmapm := \frac{\lambdapm+1}{1-\eta}
    \mymbox{and}
    \kappa := \frac{\lambdam}{\lambdap} = \frac{ \vep + 1 + \sqrt{ \rb{ \vep + 1 }^2 - 4 \vep \eta } }{ \vep + 1 - \sqrt{ \rb{ \vep + 1 }^2 - 4 \vep \eta } }.
\]
Here, $\sigmap$ and $\sigmam$ are, respectively, the slopes of the slow and fast eigenvectors of $\mat{A}$. It can be shown ${\sigmam<0}$, ${1<\sigmap<\fracslash{1}{\rb{1-\eta}}}$, and ${\kappa > \max{}{\vep,\vep^{-1}} \geq 1}$.

We want to apply Theorem~\ref{thm004}. Consider the matrices
\[
    \mat{P} := \twobytwomatrix{1}{1}{\sigmap}{\sigmam},
    \quad
    \mat{P}^{-1} = \frac{1}{\sigmap-\sigmam} \twobytwomatrix{-\sigmam}{1}{\sigmap}{-1},
    \mymbox{and}
    \matLambda := \mat{P}^{-1} \mat{A} \vec{P} = \twobytwomatrix{\lambdap}{0}{0}{\lambdam}.
\]
If we define ${\vec{u}:=\mat{P}^{-1}\vec{x}}$,
${\vec{u}_0:=\mat{P}^{-1}\vec{x}_0}$, and
${\vecf{r}{\vec{u}}:=\mat{P}^{-1}\vecf{b}{\mat{P}\vec{u}}}$, then
the initial value problem \eqref{eq068} is transformed into an
initial value problem of the form considered in
\S\ref{sec003.004.002}. Note that
\[
    \vecf{r}{\vec{u}} = \frac{ \sigmap u_1^2 + \rb{ \sigmap + \sigmam } u_1 u_2 + \sigmam u_2^2 }{ \sigmap - \sigmam } \twobyonematrix{ -\sigmam - \vep^{-1} }{ \sigmap + \vep^{-1} }
    \mymbox{and}
    r_{211} = \frac{ \sigmap \rb{ \sigmap + \vep^{-1} } }{ \sigmap - \sigmam },
\]
where $r_{211}$ is the coefficient of the $u_1^2$ term in the second
component of $\vecf{r}{\vec{u}}$.

The traditional approach of finding the behaviour of a scalar
solution $y(x)$ to \eqref{eq067} as ${x \to 0^+}$ is to obtain a
power series solution.  Using the ansatz ${y(x)=\sum_{n=0}^\infty
\sigma_n x^n}$ in \eqref{eq067} yields
\begin{equation} \label{eq069}
\sigma_n =
    \begin{cases}
        0, \mymbox{if} n = 0 \\
        \sigmap, \mymbox{if} n = 1 \\
        \ds{ -\frac{ \sum_{k=2}^{n-1} \sqb{ \rb{n\!-\!k} \sigma_{n-k} \! + \! \rb{1\!-\!\eta} \rb{n\!-\!k\!+\!1} \sigma_{n-k+1} } \sigma_k \! + \! \sqb{ \rb{n\!-\!1} \sigma_1 \! + \! \vep^{-1} } \sigma_{n-1} }{ \vep^{-1} \! + \! \rb{1\!-\!\eta} \rb{n\!+\!1} \sigma_1 \! - \! n } }, \mymbox{if} n \geq 2
    \end{cases}.
\end{equation}
One problem with using the resulting series is that the denominator
in the recursive expression for $\sigma_n$ in \eqref{eq069} is zero
if and only if ${n=\kappa}$. The authors showed this in
\cite{CalderSiegel}. This issue is resolved in the following
theorem, which generalizes Lemma~22 of the same paper.

\begin{thm} \label{thm006}
Let $\vec{x}(t)$ be a solution to the initial value problem
\eqref{eq068}, where $\norm{\vec{x}_0}$ is sufficiently small, which
approaches the origin in the slow direction from the right. Let
$y(x)$ be the corresponding scalar solution to \eqref{eq067} and
consider the coefficients $\seq{\sigma}{n}{1}{\infty}$ given in
\eqref{eq069}. Then, we can write
\[
    y(x) =
    \begin{cases}
        \ds{ \sum_{n=1}^{\floor{\kappa}} \sigma_n x^n + C \, x^\kappa + \litO{ x^\kappa } }, \mymbox{if} \kappa\not\in\setb{2,3,\ldots} \\
        \ds{ \sigmap \, x + \sqb{ \frac{ \sigmap \rb{ \sigmap + \vep^{-1} } }{ -\lambdap } } x^2 \, \lnb{x} + C \, x^2 + \litO{ x^2 } }, \mymbox{if} \kappa = 2 \\
        \ds{ \sum_{n=1}^{\kappa-1} \sigma_n x^n + C \, x^\kappa + \bigO{x^{\kappa+1}} }, \mymbox{if} \kappa \in \setb{3,4,\ldots}
    \end{cases}
    \mymbox{as}
    x \to 0^+,
\]
where $C$ is some constant (which depends on the differential
equation and the initial condition).
\end{thm}

\begin{proof}
The first part was addressed in Lemma~22 of the aforementioned
paper. Note the coefficients of the integer powers of $x$ must
indeed be $\seq{\sigma}{n}{1}{\floor{\kappa}}$ since they are
generated uniquely by the differential equation. The second and
third parts are immediate consequences of Theorem~\ref{thm004}.
\end{proof}

The authors showed, in \cite{CalderSiegel}, that for any ${\vep>0}$
there is a unique attracting solution ${y=\cM(x)}$ (the slow
manifold) to \eqref{eq067} which satisfies
\[
    H(x) < \cM(x) < \alpha(x)
    \mymbox{for all}
    x > 0,
    \mymbox{where}
    H(x) := \frac{x}{1+x}
    \mymbox{and}
    \alpha(x) := \frac{x}{\sigmap^{-1}+x}
\]
are, respectively, the horizontal isocline (the quasi-steady-state
manifold) and the isocline for slope $\sigmap$. Theorem~\ref{thm006}
tells us the asymptotic behaviour of the slow manifold (and other
solutions) at the origin up to the term for which solutions are
distinguished by initial conditions.

Planar solutions approach the slow manifold exponentially fast and
then follow it as time proceeds. Hence, it is of value to study the
dynamics on the slow manifold which is of lesser dimension than the
original system. Moreover, knowing the position of the slow manifold
accurately allows, for example, all rate constants to be determined
independently from steady-state data \cite{RousselFraser}.

\begin{rem} \label{rem002}
We can use Theorem~\ref{thm006} to comment on the validity of two
hyperbolic rate laws, specifically the quasi-steady-state
approximation $H(x)$ and the $\alpha$-approximation $\alpha(x)$. For
any physically-relevant initial condition $\vec{x}_0$ (that is,
${x_0,y_0 \geq 0}$ and ${\vec{x}_0\ne\vec{0}}$) to the initial value
problem \eqref{eq068}, as shown in \cite{CalderSiegel}, the solution
$\vec{x}(t)$ satisfies ${\vec{x}(t)\to\vec{0}}$ as ${t\to\infty}$
with the approach being in the slow direction from the right. Since
${H(x)=x+\bigO{x^2}}$ and ${\alpha(x) = \sigmap x + \bigO{x^2}}$ as
${x \to 0^+}$ and ${x(t)=\bigO{\exp{\lambdap t}}}$ as
${t\to\infty}$, it follows from Theorem~\ref{thm006} that, in terms
of the dimensional quantities of the original planar reduction
\eqref{eq064}, if ${s_0,c_0 \geq 0}$ and ${\rb{s_0,c_0}\ne\rb{0,0}}$
then
\[
    c(\tau)
    = \frac{ e_0 s(\tau) }{ K_m + s(\tau) } + \bigO{ \exp{ \sqb{ k_1 e_0 \lambdap } \tau } }
    = \frac{ e_0 s(\tau) }{ K_m \, \sigmap^{-1} + s(\tau) } +
    \begin{cases}
        \bigO{ \exp{ \sqb{ k_1 e_0 \lambdam } \tau } }, \mymbox{if} 1 < \kappa < 2 \\
        \bigO{ \tau \, \exp{ \sqb{ 2 k_1 e_0 \lambdap } \tau } }, \mymbox{if} \kappa = 2 \\
        \bigO{ \exp{ \sqb{ 2 k_1 e_0 \lambdap } \tau } }, \mymbox{if} \kappa > 2
    \end{cases}
\]
as ${\tau\to\infty}$. In all cases of the size of $\kappa$, the $\alpha$-approximation is
better than the quasi-steady-state approximation.  (In fact, there
is no better hyperbolic rate law than the $\alpha$-approximation.)
\end{rem}

\section{Open Questions} \label{sec005}

It would be useful to allow the origin to be degenerate, that is,
replace ${\mu_1<0}$ in \eqref{eq003} with ${\mu_1 \leq 0}$. This
change would not be trivial because, generally, the origin would not
be asymptotically stable (there may be hyperbolic sectors) and there
would not be exponential decay for the flow. Moreover, it would be
an interesting exercise to specify the asymptotic expansion of
solutions for a specific system when the initial condition is on an
invariant manifold. Finally, similar to Remark~\ref{rem002}, the
rate of exponential-time decay for the error in the semi-infinite
time interval version of the Quasistatic-State Approximation Theorem
in \S8.3 of \cite{Hoppensteadt} should be determinable.

\section*{Acknowledgements}

\addcontentsline{toc}{section}{Acknowledgements}

Most of the material in this paper, along with additional examples
and details, can be found in \cite{CalderThesis}, which is one of
the authors' (Calder) Ph.D. thesis written under the supervision of
the other author (Siegel). Moreover, the authors would like to
acknowledge Martin Wainwright whom once worked as a summer
undergraduate research assistant with one of the authors (Siegel).
Special cases of a few results in this paper were explored by Martin
Wainwright during this summer research term.


\begin{thebibliography}{99}

\addcontentsline{toc}{section}{References}

\bibitem{Apostol}
    T.M. Apostol, \emph{Mathematical Analysis: A Modern Approach to Advanced
    Calculus},  Addison-Wesley Publishing, Reading, Massachusetts,
    1957.

\bibitem{Bihari}
    I. Bihari, A generalization of a lemma of Bellman and its application to
    uniqueness problems of differential
    equations, \emph{Acta Math. Hungar.} \textbf{7} (1956), 81--94.

\bibitem{Bonckaert}
    P. Bonckaert, On the size of the domain of linearization in Hartman's
    theorem, \emph{J. Comput. Appl. Math.}
    \textbf{19} (1987), 279--282.

\bibitem{BronsteinKopanskii}
    I.U. Bronstein and A.Ya. Kopanskii, \emph{Smooth Invariant Manifolds and Normal
    Forms}, World Scientific Publications, Singapore, 1994.

\bibitem{CalderThesis}
    M.S. Calder, \emph{Dynamical Systems Methods Applied to the Michaelis-Menten and Lindemann
    Mechanisms}, Ph.D. Thesis, Department of Applied Mathematics, University of Waterloo, 2009.

\bibitem{CalderSiegel}
    M.S. Calder and D. Siegel, Properties of the Michaelis-Menten mechanism in phase
    space, \emph{J. Math. Anal. Appl.} \textbf{339} (2008), 1044--1064.

\bibitem{ChiconeSwanson}
    C. Chicone and R. Swanson, Linearization via the
    Lie derivative, \emph{Electron. J. Differ. Equ.} Monogr.
    \textbf{02} (2000).

\bibitem{ChowLiWang}
    S. Chow, C. Li, and D. Wang, \emph{Normal Forms and
    Bifurcation of Planar Vector Fields}, Cambridge University
    Press, New York, 1994.

\bibitem{CoddingtonLevinson}
    E.A. Coddington and N. Levinson, \emph{Theory of
    Ordinary Differential Equations}, McGraw-Hill, New York, 1955.

\bibitem{Cornish-Bowden}
    A. Cornish-Bowden, \emph{Fundamentals of Enzyme Kinetics,
    Third Edition}, Portland Press, London, 2004.

\bibitem{Grobman}
    D. Grobman, Homeomorphisms of systems of differential
    equations, \emph{Dokl. Akad. Nauk. SSSR} \textbf{128} (1959), 880--881.

\bibitem{Hartman1}
    P. Hartman, A lemma in the theory of structural
    stability of differential equations, \emph{Proc. Amer. Math.
    Soc.} \textbf{11} (1960), 610--620.

\bibitem{Hartman2}
    P. Hartman, On local homeomorphisms of Euclidean
    spaces, \emph{Bol. Soc. Mat. Mex.} \textbf{5} (1960), 220--241.

\bibitem{Hartman3}
    P. Hartman, On the local linearization of
    differential equations, \emph{Proc. Amer. Math. Soc.} \textbf{14} (1963), 568--573.

\bibitem{Henri}
    V. Henri, Th\'{e}orie g\'{e}n\'{e}rale de l'action de quelques
    diastases, \emph{C.R. Acad. Sci. Paris} \textbf{135} (1902), 916--919.

\bibitem{Hoppensteadt}
    F. Hoppensteadt, \emph{Analysis and Simulation of Chaotic
    Systems, Second Edition}, Springer-Verlag, New York, 2000.

\bibitem{KeenerSneyd}
    J. Keener and J. Sneyd, \emph{Mathematical Physiology},
    Springer, New York, 1998.

\bibitem{LaidlerBunting}
    K.J. Laidler and P.S. Bunting, \emph{The Chemical
    Kinetics of Enzyme Action, Second Edition}, Oxford University Press,
    London, 1973.

\bibitem{LevinsonRedheffer}
    N. Levinson and R.M. Redheffer, \emph{Complex
    Variables}, Holden-Day, San Francisco, 1970.

\bibitem{MichaelisMenten}
    L. Michaelis and M.L. Menten, The kinetics of the inversion effect,
    \emph{Biochem. Z.} \textbf{49} (1913), 333--369.

\bibitem{Perko}
    L. Perko, \emph{Differential Equations and Dynamical Systems, Third Edition}, Springer, New York, 2001.

\bibitem{Poincare}
    H. Poincar\'{e}, Sur les propri\'{e}t\'{e}s des fonctions d\'{e}finies par les
    \'{e}quations aux diff\'{e}rences partielles, in: \emph{Oeuvres de Henri Poincar\'{e}, Tome I},  Gauthier-Villars, Paris, 1929.

\bibitem{RousselFraser}
    M.R. Roussel and S.J. Fraser, Accurate steady-state
    approximations: implications for kinetics experiments and mechanism,
    \emph{J. Phys. Chem.} \textbf{95} (1991), 8762--8770.

\bibitem{Wiggins}
    S. Wiggins, \emph{Introduction to Applied Nonlinear Dynamical Systems and Chaos, Second Edition},  Springer-Verlag, New York, 2003.

\end{thebibliography}
\end{document}